\documentclass[12pt]{amsart}%
\usepackage{amsfonts}
\usepackage{amsmath}
\usepackage{amssymb}
\usepackage{graphicx}%
\providecommand{\U}[1]{\protect\rule{.1in}{.1in}}
\newtheorem{theorem}{Theorem}
\theoremstyle{plain}

\newtheorem{corollary}{Corollary}

\newtheorem{lemma}{Lemma}

\newtheorem{proposition}{Proposition}
\newtheorem{remark}{Remark}

\numberwithin{equation}{section}
\numberwithin{equation}{section}
\numberwithin{theorem}{section}
\numberwithin{lemma}{section}
\numberwithin{remark}{section}
\numberwithin{example}{section}
\numberwithin{proposition}{section}
\numberwithin{definition}{section}
\numberwithin{corollary}{section}
\begin{document}
\title[Taylor spectrum]{Taylor spectrum of a Banach module over\\the quantum plane}
\author{Anar Dosi}
\address{College of Mathematical Sciences, Harbin Engineering University, Nangang
District, Harbin, 150001, China}
\email{(dosiev@yahoo.com), (dosiev@metu.edu.tr)}
\date{August 21, 2024}

\begin{abstract}
In the paper we investigate the joint spectra of Banach space representations
of the quantum $q$-plane called Banach $q$-modules. Based on the
transversality relation from the topological homology of the trivial modules
versus given a left Banach $q$-module, we introduce the joint (essential)
spectra of a Banach $q$-module. In particular, we have the well defined Taylor
joint spectrum of a Banach $q$-module. The noncommutative projection
$q$-property is proved for the Taylor spectrum, which stands out the
conventional projection property in the commutative case. It is provided the
key examples of the Banach $q$-modules, which do not possesses nether forward
nor backward projection properties.

\end{abstract}
\maketitle

\section{Introduction\label{secInt}}

The present paper is devoted to the noncommutative spectral theory of Banach
space representations of Yu. I. Manin's quantum plane \cite{M}. Noncommutative
complex analytic geometry of a contractive quantum plane was investigated in
\cite{Dosi24}. The quantum plane (or just $q$-plane) is the free associative
algebra
\[
\mathfrak{A}_{q}=\mathbb{C}\left\langle x,y\right\rangle /\left(
xy-q^{-1}yx\right)
\]
generated by $x$ and $y$ modulo the relation $xy=q^{-1}yx$, where
$q\in\mathbb{C}\backslash\left\{  0,1\right\}  $, whose Arens-Michael envelope
is denoted by $\mathcal{O}_{q}\left(  \mathbb{C}_{xy}\right)  $. The
Fr\'{e}chet $\widehat{\otimes}$-algebra $\mathcal{O}_{q}\left(  \mathbb{C}%
_{xy}\right)  $ is treated as the algebra of all noncommutative entire
functions in the generators $x$ and $y$. In the case of $\left\vert
q\right\vert \neq1$, the algebra $\mathcal{O}_{q}\left(  \mathbb{C}%
_{xy}\right)  $ can be extended up to a noncommutative Fr\'{e}chet
$\widehat{\otimes}$-algebra presheaf $\mathcal{O}_{q}$ on the space
$\mathbb{C}_{xy}=\operatorname{Spec}\left(  \mathcal{O}_{q}\left(
\mathbb{C}_{xy}\right)  \right)  $ of all continuous characters of
$\mathcal{O}_{q}\left(  \mathbb{C}_{xy}\right)  $. The obtained ringed
(noncommutative) space $\left(  \mathbb{C}_{xy},\mathcal{O}_{q}\right)  $
represents the geometry of $\mathfrak{A}_{q}$ as a union of two irreducible
components being copies of the complex plane equipped with the $q$-topology
and the disk topology, respectively. Namely, $\mathbb{C}_{xy}=\mathbb{C}%
_{x}\cup\mathbb{C}_{y}$ is the union of the complex planes $\mathbb{C}%
_{x}=\mathbb{C\times}\left\{  0\right\}  \subseteq\mathbb{C}^{2}$ and
$\mathbb{C}_{y}=\left\{  0\right\}  \times\mathbb{C\subseteq C}^{2}$. The
Fr\'{e}chet $\widehat{\otimes}$-algebra $\mathcal{O}_{q}\left(  \mathbb{C}%
_{xy}\right)  $ itself admits the following topological direct sum
decomposition
\[
\mathcal{O}_{q}\left(  \mathbb{C}_{xy}\right)  =\mathcal{O}\left(
\mathbb{C}_{x}\right)  \oplus\operatorname{Rad}\mathcal{O}_{q}\left(
\mathbb{C}_{xy}\right)  \oplus\mathcal{I}_{y},
\]
where $\mathcal{O}\left(  \mathbb{C}_{x}\right)  $ is the commutative algebra
of all entire function in one variable $x$ (which is the closed unital
subalgebra generated by $x$), $\operatorname{Rad}\mathcal{O}_{q}\left(
\mathbb{C}^{2}\right)  $ is the Jacobson radical, and $\mathcal{I}_{y}$ is the
closed non-unital subalgebra generated by $y$. That decomposition of
$\mathcal{O}_{q}\left(  \mathbb{C}_{xy}\right)  $ is extended up to the sheaf
decomposition $\mathcal{O}_{q}=\mathcal{O}\oplus\operatorname{Rad}%
\mathcal{O}_{q}\oplus\mathcal{I}_{y}$ (see \cite{Dosi24}), that describes the
structure (pre)sheaf $\mathcal{O}_{q}$.

Our main task is to analyze the Taylor type spectrum of a Banach space
representation of $\mathfrak{A}_{q}$ that comes from the topological homology
sight of $\mathcal{O}_{q}\left(  \mathbb{C}_{xy}\right)  $ obtained by A. Yu.
Prikovskii in \cite{Pir} (see also \cite{Pir09}). One of the central results
of \cite{Pir} asserts that the canonical embedding $\mathfrak{A}%
_{q}\rightarrow\mathcal{O}_{q}\left(  \mathbb{C}_{xy}\right)  $ is a
localization in the sense of Taylor \cite{Tay2}. It turns out that every left
Banach $\mathfrak{A}_{q}$-module possesses the Taylor spectrum as in the
commutative case, which carries a deep information on its homological
properties and the related noncommutative holomorphic functional calculus. In
the commutative case that approach is due to M. Putinar \cite{Put}, that is
the case of $q=1$.

A left Banach $\mathfrak{A}_{q}$-module $X$ is given by a complex Banach space
$X$ and an ordered pair $\left(  T,S\right)  $ of bounded linear operators
from $\mathcal{B}\left(  X\right)  $ such that $TS=q^{-1}ST$. The pair
$\left(  T,S\right)  $ automatically defines a unique continuous algebra
homomorphism $\mathcal{O}_{q}\left(  \mathbb{C}_{xy}\right)  \rightarrow
\mathcal{B}\left(  X\right)  $ that turns $X$ into a left Banach
$\mathcal{O}_{q}\left(  \mathbb{C}_{xy}\right)  $-module. The resolvent
set\textit{ }$\operatorname{res}\left(  T,S\right)  $ of these operators (or
the module $X$) is defined as a set of those $\gamma\in\mathbb{C}_{xy}$ such
that the transversality relation $\mathbb{C}\left(  \gamma\right)
\perp_{\mathcal{O}_{q}\left(  \mathbb{C}_{xy}\right)  }X$ holds, where
$\mathbb{C}\left(  \gamma\right)  $ is the trivial $\mathcal{O}_{q}\left(
\mathbb{C}_{xy}\right)  $-module given by the continuous character $\gamma$.
Recall that $\mathbb{C}\left(  \gamma\right)  \perp_{\mathcal{O}_{q}\left(
\mathbb{C}_{xy}\right)  }X$ means vanishing of all homology groups
$\operatorname{Tor}_{k}^{\mathcal{O}_{q}\left(  \mathbb{C}_{xy}\right)
}\left(  \mathbb{C}\left(  \gamma\right)  ,X\right)  =\left\{  0\right\}  $,
$k\geq0$. The joint (Taylor) spectrum of the pair $\left(  T,S\right)  $ or
the left Banach $\mathcal{O}_{q}\left(  \mathbb{C}_{xy}\right)  $-module $X$
is defined as the complement $\sigma\left(  T,S\right)  =\mathbb{C}%
_{xy}\backslash\operatorname{res}\left(  T,S\right)  $ to the resolvent set.
It turns out that $\gamma\in\sigma\left(  T,S\right)  $ if and only if the
following Banach space complex $\mathcal{K}\left(  \left(  T,S\right)
,\gamma\right)  :$
\[%
\begin{tabular}
[c]{lllllll}%
$0\rightarrow$ & $X$ & $\overset{\left[
\begin{tabular}
[c]{l}%
$\gamma\left(  y\right)  -qS$\\
$T-q\gamma\left(  x\right)  $%
\end{tabular}
\ \ \right]  }{\longrightarrow}$ & $X\oplus X$ & $\overset{\left[
\begin{tabular}
[c]{ll}%
$T-\gamma\left(  x\right)  $ & $S-\gamma\left(  y\right)  $%
\end{tabular}
\ \ \ \right]  }{\longrightarrow}$ & $X$ & $\rightarrow0$%
\end{tabular}
\ \
\]
fails to be exact, that is,
\[
\sigma\left(  T,S\right)  =\left\{  \gamma\in\mathbb{C}_{xy}:\mathcal{K}%
\left(  \left(  T,S\right)  ,\gamma\right)  \text{ is not exact}\right\}  .
\]
Since $\mathcal{K}\left(  \left(  T,S\right)  ,\gamma\right)  $, $\gamma
\in\mathbb{C}_{xy}$ is a parametrized Banach space complex, all properties of
the Slodkowski spectra (see \cite{DosAJM}) and the essential spectrum
\cite{Fain80}, \cite{Cur} of parametrized Banach space complexes are
applicable. We prove that if $T$ and $S$ are invertible operators, then
$\sigma\left(  T,S\right)  =\varnothing$ and $\left\vert q\right\vert =1$. If
$\left\vert q\right\vert \neq1$ then $\sigma\left(  T,S\right)  $ is a
nonempty compact subset in $\mathbb{C}^{2}$ (in the standard topology). As the
main result we prove the following $q$-projection property
\[
\sigma\left(  T,S\right)  \subseteq\left(  \left(  \sigma\left(  T\right)
\cup\sigma\left(  q^{-1}T\right)  \right)  \times\left\{  0\right\}  \right)
\cup\left(  \left\{  0\right\}  \times\left(  \sigma\left(  S\right)
\cup\sigma\left(  qS\right)  \right)  \right)  .
\]
Moreover, we show that the standard forward $\sigma\left(  T,S\right)
|\mathbb{C}_{x}\subseteq\sigma\left(  T\right)  \times\left\{  0\right\}  $
and the backward $\left\{  0\right\}  \times\sigma\left(  S\right)
\subseteq\sigma\left(  T,S\right)  |\mathbb{C}_{y}$ projection properties do
not hold in the general case, that is surprising result compared to the
noncommutative Harte spectrum \cite{Dosi24} with its noncommutative polynomial
spectral mapping property.

As a key example we consider a pair of the unilateral shift operator $T$ and
the diagonal $q$-operator $S$ on the Banach spaces $\ell_{p}$, $1\leq
p<\infty$, and provide the full description of their spectrum $\sigma\left(
T,S\right)  $ and the essential spectrum $\sigma_{e}\left(  T,S\right)  $. In
this case, $\sigma\left(  T\right)  =\mathbb{D}_{1}$ is the closed unit disk
centered at the origin and $\sigma\left(  S\right)  =\left\{  q^{n}%
:n\geq0\right\}  ^{-}$. It turns out that
\[
\left(  \mathbb{D}_{\left\vert q\right\vert ^{-1}}\backslash\mathbb{D}%
_{1}^{\circ}\right)  \times\left\{  0\right\}  \subseteq\sigma\left(
T,S\right)  |\mathbb{C}_{x}\subseteq\mathbb{D}_{\left\vert q\right\vert ^{-1}%
}\times\left\{  0\right\}  ,\quad\sigma\left(  T,S\right)  |\mathbb{C}%
_{y}=\left\{  \left(  0,1\right)  \right\}  ,
\]
where $\mathbb{D}_{r}$ is the closed unit disk of radius $r$ centered at the
origin with its interior $\mathbb{D}_{r}^{\circ}$ and the topological boundary
$\partial\mathbb{D}_{r}$. Thus $\sigma\left(  T,S\right)  |\mathbb{C}%
_{x}\nsubseteqq\sigma\left(  T\right)  \times\left\{  0\right\}  $ and
$\left\{  0\right\}  \times\sigma\left(  S\right)  \nsubseteqq\sigma\left(
T,S\right)  |\mathbb{C}_{y}$.

Finally, for the joint essential spectrum we obtain the following formula
\[
\sigma_{e}\left(  T,S\right)  =\left(  \partial\mathbb{D}_{1}\cup
\partial\mathbb{D}_{\left\vert q\right\vert ^{-1}}\right)  \times\left\{
0\right\}
\]
though $\sigma_{e}\left(  T\right)  =\partial\mathbb{D}_{1}$, which shows
again the failure of the conventional projection property for the joint
essential spectrum.

\section{Preliminaries\label{sPre}}

All considered vector spaces are assumed to be complex. The spectrum of an
element $a$ in an associative algebra $A$ is denoted by $\sigma\left(
a\right)  $, whereas $\operatorname{Rad}A$ denotes the Jacobson radical of $A$.

\subsection{Arens-Michael envelope}

Let $A$ be a complete polynormed (or locally convex) algebra. If the topology
of $A$ is defined by means of a family of multiplicative seminorms, then $A$
is called an \textit{Arens-Michael algebra} \cite[1.2.4]{Hel}. Fix a
polynormed algebra $A$ with its separately continuous multiplication. The
\textit{Arens-Michael envelope} \cite[5.2.21]{Hel} of $A$ is called a pair
$\left(  \widetilde{A},\omega\right)  $, where $\widetilde{A}$ is an
Arens-Michael algebra, $\omega:A\rightarrow\widetilde{A}$ is a continuous
algebra homomorphism with the following \textquotedblright
universal-projective\textquotedblright\ property: for all Arens-Michael
algebra $B$ and a continuous algebra homomorphism $\pi:A\rightarrow B$, there
exists a unique continuous algebra homomorphism $\widetilde{\pi}%
:\widetilde{A}\rightarrow B$ such that $\widetilde{\pi}\cdot\omega=\pi$. It
turns out that an Arens-Michael algebra is an inverse limit of some Banach
algebras \cite[5.2.10]{Hel}. Thus in the latter universal projective property
it can be assumed that all considered algebras $B$ are Banach algebras.

The set of all continuous characters of an Arens-Michael algebra $A$ is
denoted by $\operatorname{Spec}\left(  A\right)  $. If $\lambda\in
\operatorname{Spec}\left(  A\right)  $ then the algebra homomorphism
$\lambda:A\rightarrow\mathbb{C}$ defines $A$-module structure on $\mathbb{C}$
via pull back along $\lambda$. This module is denoted by $\mathbb{C}\left(
\lambda\right)  $ called \textit{a trivial module given by }$\lambda$.

\subsection{The Fr\'{e}chet algebra $\mathcal{O}_{q}\left(  \mathbb{C}%
_{xy}\right)  $\label{sAME}}

The Arens-Michael envelope of the $q$-plane $\mathfrak{A}_{q}$ is denoted by
$\mathcal{O}_{q}\left(  \mathbb{C}^{2}\right)  $ (see \cite{Pir}). It turns
out that $\mathcal{O}_{q}\left(  \mathbb{C}^{2}\right)  =\mathcal{I}_{x}%
\oplus\mathcal{I}_{xy}\oplus\mathcal{I}_{y}$ is a topological direct sum of
the closed unital subalgebra $\mathcal{I}_{x}$ generated by $x$, the closed
two-sided ideal $\mathcal{I}_{xy}$ generated by $xy$, and the closed
non-unital subalgebra $\mathcal{I}_{y}$ generated by $y$ (see \cite{Dosi24}).
In this case, $\operatorname{Rad}\mathcal{O}_{q}\left(  \mathbb{C}^{2}\right)
\subseteq\mathcal{I}_{xy}=\cap\left\{  \ker\left(  \lambda\right)  :\lambda
\in\operatorname{Spec}\left(  \mathcal{O}_{q}\left(  \mathbb{C}^{2}\right)
\right)  \right\}  $ and%
\[
\operatorname{Spec}\left(  \mathcal{O}_{q}\left(  \mathbb{C}^{2}\right)
\right)  =\operatorname{Spec}\left(  \mathcal{O}_{q}\left(  \mathbb{C}%
^{2}\right)  /\mathcal{I}_{xy}\right)  =\mathbb{C}_{xy},
\]
where $\mathbb{C}_{xy}=\mathbb{C}_{x}\cup\mathbb{C}_{y}$, $\mathbb{C}%
_{x}=\mathbb{C\times}\left\{  0\right\}  \subseteq\mathbb{C}^{2}$,
$\mathbb{C}_{y}=\left\{  0\right\}  \times\mathbb{C\subseteq C}^{2}$. If
$\left\vert q\right\vert \neq1$ then the algebra $\mathcal{O}_{q}\left(
\mathbb{C}^{2}\right)  $ is commutative modulo its Jacobson radical
$\operatorname{Rad}\mathcal{O}_{q}\left(  \mathbb{C}^{2}\right)  $ and
$\mathcal{I}_{xy}=\operatorname{Rad}\mathcal{O}_{q}\left(  \mathbb{C}%
^{2}\right)  $, that is, the topological direct sum decomposition
\[
\mathcal{O}_{q}\left(  \mathbb{C}^{2}\right)  =\mathcal{O}\left(
\mathbb{C}_{x}\right)  \oplus\operatorname{Rad}\mathcal{O}_{q}\left(
\mathbb{C}^{2}\right)  \oplus\mathcal{I}_{y}%
\]
holds. Thus the noncommutative analytic space $\mathbb{C}_{xy}=\mathbb{C}%
_{x}\cup\mathbb{C}_{y}$ is a union of two copies of the complex plane. One
needs to equip $\mathbb{C}_{xy}$ with a suitable topology that would affiliate
to the noncommutative multiplication of the quantum plane (see \cite{Dosi24}).
The algebra $\mathcal{O}_{q}\left(  \mathbb{C}^{2}\right)  $ turns out to be
commutative modulo its Jacobson radical $\operatorname{Rad}\mathcal{O}%
_{q}\left(  \mathbb{C}^{2}\right)  $ and $\mathcal{I}_{xy}=\operatorname{Rad}%
\mathcal{O}_{q}\left(  \mathbb{C}^{2}\right)  $. In this case, we use the
notation $\mathcal{O}_{q}\left(  \mathbb{C}_{xy}\right)  $ instead of
$\mathcal{O}_{q}\left(  \mathbb{C}^{2}\right)  $ by indicating to the fact
that it can be extended up to a Fr\'{e}chet $\widehat{\otimes}$-algebra
presheaf $\mathcal{O}_{q}$ on the space $\mathbb{C}_{xy}$ \cite{Dosi24}. If
$\left\vert q\right\vert =1$ then none of the monomials $x^{i}y^{k}%
\in\mathcal{I}_{xy}$, $i,k\in\mathbb{N}$ is quasinilpotent, whereas
$\mathcal{I}_{xy}$ contains all quasinilpotent elements.

\section{The Slodkowski spectra of left $\mathfrak{A}_{q}$-Fr\'{e}chet
modules}

In this section we introduce the Slodkowski (Taylor) spectra of the left
Banach $\mathfrak{A}_{q}$-modules on the topological homology background and
investigate the transversality relations.

\subsection{The Slodkowski, Taylor and essential spectra\label{subsecSTE}}

Let $X$ be a left Banach $\mathfrak{A}_{q}$-module given by a pair $\left(
T,S\right)  $ of operators from $\mathcal{B}\left(  X\right)  $ with the
property $TS=q^{-1}ST$. Since $\mathcal{O}_{q}\left(  \mathbb{C}^{2}\right)  $
is the Arens-Michael envelope of $\mathfrak{A}_{q}$, it follows that the same
couple $\left(  T,S\right)  $ provides a left Banach $\mathcal{O}_{q}\left(
\mathbb{C}^{2}\right)  $-module structure on $X$ extending the original one.
Let us introduce the following parametrized over the character space
$\mathbb{C}_{xy}$ Banach space complex
\begin{equation}
\mathcal{K}\left(  \left(  T,S\right)  ,\gamma\right)  :0\rightarrow
X\overset{d_{\gamma}^{0}}{\longrightarrow}X\oplus X\overset{d_{\gamma}%
^{1}}{\longrightarrow}X\rightarrow0 \label{Xkos}%
\end{equation}
with the differentials
\[
d_{\gamma}^{0}=\left[
\begin{tabular}
[c]{l}%
$\gamma\left(  y\right)  -qS$\\
$T-q\gamma\left(  x\right)  $%
\end{tabular}
\ \right]  \quad\text{and}\quad d_{\gamma}^{1}=\left[
\begin{tabular}
[c]{ll}%
$T-\gamma\left(  x\right)  $ & $S-\gamma\left(  y\right)  $%
\end{tabular}
\ \right]  .
\]
Notice that $\gamma\left(  x\right)  \gamma\left(  y\right)  =0$ for every
$\gamma\in\mathbb{C}_{xy}$. One can consider the related Slodkowski spectra
$\sigma^{\pi,n}\left(  T,S\right)  $, $\sigma^{\delta,n}\left(  T,S\right)  $,
$0\leq n\leq2$, and the essential spectrum $\sigma_{e}\left(  T,S\right)  $
(see \cite{Tay}, \cite{DosJOT1} and \cite{DosAJM}) of the parametrized Banach
space complex (\ref{Xkos}). Namely, put $\Sigma^{n}\left(  T,S\right)
=\left\{  \gamma\in\mathbb{C}_{xy}:H^{n}\left(  \gamma\right)  \neq\left\{
0\right\}  \right\}  $, where $H^{n}\left(  \gamma\right)  $ is the $n$th
cohomology group of the complex $\mathcal{K}\left(  \left(  T,S\right)
,\gamma\right)  $. Recall that
\[
\sigma^{\pi,n}\left(  T,S\right)  =\left\{  \gamma\in\bigcup\limits_{k\leq
n}\Sigma^{n}\left(  T,S\right)  \text{ or }\operatorname{im}\left(  d_{\gamma
}^{n}\right)  \text{ is not closed}\right\}  ,\text{ and }\sigma^{\delta
,n}\left(  T,S\right)  =\bigcup\limits_{k\geq n}\Sigma^{n}\left(  T,S\right)
\text{. }%
\]
\textit{The joint spectrum of the pair }$\left(  T,S\right)  $ is defined to
be the following set
\[
\sigma\left(  T,S\right)  =\left\{  \gamma\in\mathbb{C}_{xy}:\mathcal{K}%
\left(  \left(  T,S\right)  ,\gamma\right)  \text{ is not exact}\right\}  ,
\]
whereas \textit{the joint essential spectrum of the pair} $\left(  T,S\right)
$ is defined as its subset
\[
\sigma_{e}\left(  T,S\right)  =\left\{  \gamma\in\mathbb{C}_{xy}%
:\mathcal{K}\left(  \left(  T,S\right)  ,\gamma\right)  \text{ is not
Fredholm}\right\}  .
\]
Recall that a Banach space complex is said to be Fredholm if all its
(co)homology groups are finite dimensional. Notice that
\[
\sigma\left(  T,S\right)  =\bigcup\limits_{k}\Sigma^{k}\left(  T,S\right)
=\sigma^{\pi,2}\left(  T,S\right)  =\sigma^{\delta,0}\left(  T,S\right)  .
\]
It is an analog of the Taylor joint spectrum of the operator pair $\left(
T,S\right)  $, and it can also be applied to the case of a left Fr\'{e}chet
$\mathcal{O}_{q}\left(  \mathbb{C}_{xy}\right)  $-module $X$ given by a pair
$\left(  T,S\right)  $ of continuous linear operators acting on $X$ with
$TS=q^{-1}ST$. Since $\operatorname{Spec}\left(  \mathcal{O}_{q}\left(
\mathbb{C}^{2}\right)  \right)  =\mathbb{C}_{xy}=\mathbb{C}_{x}\cup
\mathbb{C}_{y}$, it follows that
\begin{align*}
\sigma\left(  T,S\right)   &  =\sigma_{x}\left(  T,S\right)  \cup\sigma
_{y}\left(  T,S\right)  \text{ with}\\
\sigma_{x}\left(  T,S\right)   &  =\sigma\left(  T,S\right)  \cap
\mathbb{C}_{x}\text{ and }\sigma_{y}\left(  T,S\right)  =\sigma\left(
T,S\right)  \cap\mathbb{C}_{y}.
\end{align*}
If $\lambda\in\sigma_{x}\left(  T,S\right)  $ then the differentials of the
complex $\mathcal{K}\left(  \left(  T,S\right)  ,\lambda\right)  $ are
converted into
\[
d_{\lambda}^{0}=\left[
\begin{tabular}
[c]{l}%
$-qS$\\
$T-q\lambda$%
\end{tabular}
\ \ \ \right]  \quad\text{and}\quad d_{\lambda}^{1}=\left[
\begin{tabular}
[c]{ll}%
$T-\lambda$ & $S$%
\end{tabular}
\ \ \ \right]  .
\]
Notice that the first (nonzero) homology group $H^{0}\left(  \lambda\right)  $
of the complex $\mathcal{K}\left(  \left(  T,S\right)  ,\lambda\right)  $
corresponds to the joint eigenvalue $\left(  q\lambda,0\right)  $ of the pair
$\left(  T,S\right)  $, whereas nonzero $H^{2}\left(  \lambda\right)  $ occurs
if $\operatorname{im}\left(  T-\lambda\right)  +\operatorname{im}\left(
S\right)  \neq X$. If $\mu\in\sigma_{y}\left(  T,S\right)  $ then we come up
with the following non-exact Banach space complex $\mathcal{K}\left(  \left(
T,S\right)  ,\mu\right)  $ whose differentials are given by
\[
d_{\mu}^{0}=\left[
\begin{tabular}
[c]{l}%
$\mu-qS$\\
$T$%
\end{tabular}
\ \ \ \ \right]  \quad\text{and}\quad d_{\mu}^{1}=\left[
\begin{tabular}
[c]{ll}%
$T$ & $S-\mu$%
\end{tabular}
\ \ \ \right]  .
\]
In this case, a nonzero $H^{0}\left(  \mu\right)  $ responds to the joint
eigenvalue $\left(  0,q^{-1}\mu\right)  $ of the operator pair $\left(
T,S\right)  $, whereas a nonzero $H^{2}\left(  \mu\right)  $ occurs if
$\operatorname{im}\left(  T\right)  +\operatorname{im}\left(  S-\mu\right)
\neq X$.

\subsection{The resolution of the quantum plane}

The quantum plane $\mathfrak{A}_{q}$ possesses (see \cite{Tah}, \cite{WM}) the
following free $\mathfrak{A}_{q}$-bimodule resolution
\begin{equation}
0\leftarrow\mathfrak{A}_{q}\overset{\pi}{\longleftarrow}\mathfrak{A}%
_{q}\otimes\mathfrak{A}_{q}\overset{\partial_{0}}{\longleftarrow}%
\mathfrak{A}_{q}\otimes\mathbb{C}^{2}\otimes\mathfrak{A}_{q}\overset{\partial
_{1}}{\longleftarrow}\mathfrak{A}_{q}\otimes\wedge^{2}\mathbb{C}^{2}%
\otimes\mathfrak{A}_{q}\leftarrow0, \label{Ures}%
\end{equation}
with the mapping $\pi\left(  a\otimes b\right)  =ab$ and the differentials
\begin{align*}
\partial_{0}\left(  a\otimes e_{1}\otimes b\right)   &  =a\otimes xb-ax\otimes
b,\quad\partial_{0}\left(  a\otimes e_{2}\otimes b\right)  =a\otimes
yb-ay\otimes b,\\
\partial_{1}\left(  a\otimes e_{1}\wedge e_{2}\otimes b\right)   &  =a\otimes
e_{2}\otimes xb-qax\otimes e_{2}\otimes b-qa\otimes e_{1}\otimes yb+ay\otimes
e_{1}\otimes b,
\end{align*}
where $a,b\in\mathfrak{A}_{q}$ and $\left(  e_{1},e_{2}\right)  $ is the
standard basis of $\mathbb{C}^{2}$. One of the central results of \cite{Pir}
asserts that the canonical embedding $\mathfrak{A}_{q}\rightarrow
\mathcal{O}_{q}\left(  \mathbb{C}^{2}\right)  $ is a localization in the sense
of Taylor \cite{Tay2}. Using the resolution (\ref{Ures}), we derive that the
following (a similar) complex%
\[
0\leftarrow\mathcal{O}_{q}\left(  \mathbb{C}^{2}\right)  \overset{\pi
}{\longleftarrow}\mathcal{O}_{q}\left(  \mathbb{C}^{2}\right)
^{\widehat{\otimes}2}\overset{d_{0}}{\longleftarrow}\mathcal{O}_{q}\left(
\mathbb{C}^{2}\right)  \widehat{\otimes}\mathbb{C}^{2}\widehat{\otimes
}\mathcal{O}_{q}\left(  \mathbb{C}^{2}\right)  \overset{d_{1}}{\longleftarrow
}\mathcal{O}_{q}\left(  \mathbb{C}^{2}\right)  \widehat{\otimes}\wedge
^{2}\mathbb{C}^{2}\widehat{\otimes}\mathcal{O}_{q}\left(  \mathbb{C}%
^{2}\right)  \leftarrow0
\]
is admissible, which provides a free $\mathcal{O}_{q}\left(  \mathbb{C}%
^{2}\right)  $-bimodule resolution of $\mathcal{O}_{q}\left(  \mathbb{C}%
^{2}\right)  $ (see \cite[3.2]{HelHom}). One needs to apply the functor
$\mathcal{O}_{q}\left(  \mathbb{C}^{2}\right)  \underset{\mathfrak{A}%
_{q}}{\widehat{\otimes}}\circ\underset{\mathfrak{A}_{q}}{\widehat{\otimes}%
}\mathcal{O}_{q}\left(  \mathbb{C}^{2}\right)  $ to the complex (\ref{Ures}).
We use a bit modified cochain version of this resolution
\[
\mathcal{R}\left(  \mathcal{O}_{q}\left(  \mathbb{C}^{2}\right)
^{\widehat{\otimes}2}\right)  :0\rightarrow\mathcal{O}_{q}\left(
\mathbb{C}^{2}\right)  ^{\widehat{\otimes}2}\overset{d^{0}}{\longrightarrow
}\mathcal{O}_{q}\left(  \mathbb{C}^{2}\right)  ^{\widehat{\otimes}2}%
\oplus\mathcal{O}_{q}\left(  \mathbb{C}^{2}\right)  ^{\widehat{\otimes}%
2}\overset{d^{1}}{\longrightarrow}\mathcal{O}_{q}\left(  \mathbb{C}%
^{2}\right)  ^{\widehat{\otimes}2}\rightarrow0,
\]
whose differentials can be written in the following operator matrix shapes
\begin{equation}
d^{0}=\left[
\begin{array}
[c]{c}%
R_{y}\otimes1-q1\otimes L_{y}\\
1\otimes L_{x}-qR_{x}\otimes1
\end{array}
\right]  ,\text{ }d^{1}=\left[
\begin{array}
[c]{cc}%
1\otimes L_{x}-R_{x}\otimes1 & 1\otimes L_{y}-R_{y}\otimes1
\end{array}
\right]  , \label{d}%
\end{equation}
where $L$ and $R$ indicate to the left and right multiplication operators on
$\mathcal{O}_{q}\left(  \mathbb{C}^{2}\right)  $.

\subsection{The spectrum of a left Fr\'{e}chet $\mathcal{O}_{q}\left(
\mathbb{C}^{2}\right)  $-module}

Let us introduce the spectrum of a left Fr\'{e}chet $\mathcal{O}_{q}\left(
\mathbb{C}^{2}\right)  $-module $X$ within the general framework of the
spectral theory of left modules over Fr\'{e}chet sheaves \cite{Dos}.

Let $X$ be a left Fr\'{e}chet $\mathcal{O}_{q}\left(  \mathbb{C}^{2}\right)
$-module given by an operator pair $\left(  T,S\right)  $ with $TS=q^{-1}ST$.
By applying the functor $\circ\underset{\mathcal{O}_{q}\left(  \mathbb{C}%
^{2}\right)  }{\widehat{\otimes}}X$ to the indicated bimodule resolution
$\mathcal{R}\left(  \mathcal{O}_{q}\left(  \mathbb{C}^{2}\right)
^{\widehat{\otimes}2}\right)  $ of $\mathcal{O}_{q}\left(  \mathbb{C}%
^{2}\right)  $, we obtain the following free left $\mathcal{O}_{q}\left(
\mathbb{C}^{2}\right)  $-module resolution
\begin{equation}
\mathcal{R}\left(  \mathcal{O}_{q}\left(  \mathbb{C}^{2}\right)
\widehat{\otimes}X\right)  :0\rightarrow\mathcal{O}_{q}\left(  \mathbb{C}%
^{2}\right)  \widehat{\otimes}X\overset{d^{0}}{\longrightarrow}\left(
\mathcal{O}_{q}\left(  \mathbb{C}^{2}\right)  \widehat{\otimes}X\right)
^{\oplus2}\overset{d^{1}}{\longrightarrow}\mathcal{O}_{q}\left(
\mathbb{C}^{2}\right)  \widehat{\otimes}X\rightarrow0 \label{reX}%
\end{equation}
for the module $X$ with the differentials (see (\ref{d}))%
\[
d^{0}=\left[
\begin{array}
[c]{c}%
R_{y}\otimes1-q1\otimes S\\
1\otimes T-qR_{x}\otimes1
\end{array}
\right]  \text{, }d^{1}=\left[
\begin{array}
[c]{cc}%
1\otimes T-R_{x}\otimes1 & 1\otimes S-R_{y}\otimes1
\end{array}
\right]  .
\]
Thus the complex
\[
\mathcal{R}\left(  \mathcal{O}_{q}\left(  \mathbb{C}^{2}\right)
\widehat{\otimes}X\right)  \overset{\pi_{X}}{\longrightarrow}X\rightarrow
0\quad\text{with }\pi_{X}\left(  f\otimes\xi\right)  =f\xi
\]
is admissible. \textit{The resolvent set }$\operatorname{res}\left(
T,S\right)  $ of these operators (or the module $X$) is defined as a set of
those $\gamma\in\mathbb{C}^{2}$ such that the transversality relation
$\mathbb{C}\left(  \gamma\right)  \perp_{\mathcal{O}_{q}\left(  \mathbb{C}%
^{2}\right)  }X$ holds, that is, $\operatorname{Tor}_{k}^{\mathcal{O}%
_{q}\left(  \mathbb{C}^{2}\right)  }\left(  \mathbb{C}\left(  \gamma\right)
,X\right)  =\left\{  0\right\}  $ for all $k\geq0$.

\begin{lemma}
\label{lemRESTS}Let $X$ be a left Fr\'{e}chet $\mathcal{O}_{q}\left(
\mathbb{C}^{2}\right)  $-module given by an operator pair $\left(  T,S\right)
$ with $TS=q^{-1}ST$. The equality
\[
\sigma\left(  T,S\right)  =\mathbb{C}^{2}\backslash\operatorname{res}\left(
T,S\right)
\]
holds.
\end{lemma}

\begin{proof}
The homology groups $\operatorname{Tor}_{k}^{\mathcal{O}_{q}\left(
\mathbb{C}^{2}\right)  }\left(  \mathbb{C}\left(  \gamma\right)  ,X\right)  $
can be calculated by means of the resolution $\mathcal{R}\left(
\mathcal{O}_{q}\left(  \mathbb{C}^{2}\right)  \widehat{\otimes}X\right)  $
from (\ref{reX}). By applying the functor $\mathbb{C}\left(  \gamma\right)
\underset{\mathcal{O}_{q}\left(  \mathbb{C}^{2}\right)  }{\widehat{\otimes}%
}\circ$ to the resolution $\mathcal{R}\left(  \mathcal{O}_{q}\left(
\mathbb{C}^{2}\right)  \widehat{\otimes}X\right)  $, we derive that
$\operatorname{Tor}_{k}^{\mathcal{O}_{q}\left(  \mathbb{C}^{2}\right)
}\left(  \mathbb{C}\left(  \gamma\right)  ,X\right)  $ are the homology groups
of the following complex
\begin{align*}
0  &  \rightarrow X\overset{d_{\gamma}^{0}}{\longrightarrow}X\oplus
X\overset{d_{\gamma}^{1}}{\longrightarrow}X\rightarrow0,\\
d_{\gamma}^{0}  &  =\left[
\begin{array}
[c]{c}%
\gamma\left(  y\right)  -qS\\
T-q\gamma\left(  x\right)
\end{array}
\right]  \text{, }d_{\gamma}^{1}=\left[
\begin{array}
[c]{cc}%
T-\gamma\left(  x\right)  & S-\gamma\left(  y\right)
\end{array}
\right]  ,
\end{align*}
which is reduced to the complex $\mathcal{K}\left(  \left(  T,S\right)
,\gamma\right)  $ considered above in Subsection \ref{subsecSTE}. Thus
$\operatorname{res}\left(  T,S\right)  $ consists of those $\gamma
\in\mathbb{C}_{xy}$ such that $\mathcal{K}\left(  \left(  T,S\right)
,\gamma\right)  $ is exact. Hence $\mathbb{C}_{xy}\backslash\operatorname{res}%
\left(  T,S\right)  =\sigma\left(  T,S\right)  $.
\end{proof}

Thus the Taylor spectrum $\sigma\left(  T,S\right)  $ is defined in terms of
the transversality relation of the trivial modules $\mathbb{C}\left(
\gamma\right)  $, $\gamma\in\mathbb{C}_{xy}$ with respect to the given left
Fr\'{e}chet $\mathcal{O}_{q}\left(  \mathbb{C}^{2}\right)  $-module.

\section{The pseudo-commuting pairs of operators}

In this section we introduce the pseudo-commuting pairs of bounded linear
operators and their joint spectra that allow to describe the Taylor spectrum
of a left Banach $\mathfrak{A}_{q}$-module. We also prove the key result on
the $q$-projection property for a left Banach $\mathfrak{A}_{q}$-module.

\subsection{The left spectrum of a pseudo-commuting pairs of operators}

Let $\left(  T,S\right)  $ be an ordered pair of bounded linear operators
acting on a Banach space $X$. We say that $\left(  T,S\right)  $ is a
\textit{pseudo-commuting pair} if $TS=SC$ for a certain $C\in\mathcal{B}%
\left(  X\right)  $. In particular, $\left(  \lambda-T\right)  S=S\left(
\lambda-C\right)  $ for all $\lambda\in\mathbb{C}$. The pseudo-commuting pair
$\left(  T,S\right)  $ generates the following (continuously) parametrized
\cite[2.1]{Tay} on $\mathbb{C}$ (cochain) Banach space complex
\[
\mathcal{L}_{\lambda}\left(  T,S\right)  :0\rightarrow X\overset{l_{\lambda
}^{0}}{\longrightarrow}X\oplus X\overset{l_{\lambda}^{1}}{\longrightarrow
}X\rightarrow0
\]
with the differentials%
\[
l_{\lambda}^{0}=\left[
\begin{tabular}
[c]{l}%
$S$\\
$\lambda-C$%
\end{tabular}
\right]  \quad\text{and}\quad l_{\lambda}^{1}=\left[
\begin{tabular}
[c]{ll}%
$\lambda-T$ & $-S$%
\end{tabular}
\right]  .
\]
Thus $l_{\lambda}^{1}l_{\lambda}^{0}=\left(  \lambda-T\right)  S-S\left(
\lambda-C\right)  =0$ for all $\lambda$. We say that $\lambda$ is \textit{a
left spectral point of the pair }$\left(  T,S\right)  $ if the complex
$\mathcal{L}_{\lambda}\left(  T,S\right)  $ fails to be exact. The set of all
left spectral points of $\left(  T,S\right)  $ is denoted by $\sigma
_{l}\left(  T,S\right)  $.

\begin{lemma}
\label{lTSsp}If $\left(  T,S\right)  $ is a pseudo-commuting pair of operators
on a Banach space $X$ with $TS=SC$, then%
\[
\sigma_{l}\left(  T,S\right)  \subseteq\sigma\left(  T\right)  \cup
\sigma\left(  C\right)  .
\]
Moreover, if $S$ is invertible in $\mathcal{B}\left(  X\right)  $ then
$\sigma_{l}\left(  T,S\right)  =\varnothing$.
\end{lemma}

\begin{proof}
If $\lambda\notin\sigma\left(  T\right)  \cup\sigma\left(  C\right)  $ then
obviously, $\ker\left(  l_{\lambda}^{0}\right)  =\ker\left(  S\right)
\cap\ker\left(  \lambda-C\right)  =\left\{  0\right\}  $ and
$\operatorname{im}\left(  l_{\lambda}^{1}\right)  =\left(  \lambda-T\right)
X+SX=X$. Next, if $\left(  \xi,\eta\right)  \in\ker\left(  l_{\lambda}%
^{1}\right)  $, then $\left(  \lambda-T\right)  \xi=S\eta$, which in turn
implies that $\xi=\left(  \lambda-T\right)  ^{-1}S\eta$. With $\left(
\lambda-T\right)  S=S\left(  \lambda-C\right)  $ in mind, infer that
$\xi=S\theta$, where $\theta=\left(  \lambda-C\right)  ^{-1}\eta$. Then
\[
l_{\lambda}^{0}\theta=\left(  S\theta,\left(  \lambda-C\right)  \theta\right)
=\left(  \xi,\eta\right)  .
\]
Thus the complex $\mathcal{L}_{\lambda}\left(  T,S\right)  $ is exact, or
$\lambda\notin\sigma_{l}\left(  T,S\right)  $. Hence $\sigma_{l}\left(
T,S\right)  \subseteq\sigma\left(  T\right)  \cup\sigma\left(  C\right)  $.

Now assume that $S$ is invertible in $\mathcal{B}\left(  X\right)  $. So,
$S^{-1}T=CS^{-1}$. If $l_{\lambda}^{0}\zeta=0$ for some $\zeta\in X$, then
$S\zeta=0$, therefore $\zeta=0$. Moreover, $l_{\lambda}^{1}\left(  X\oplus
X\right)  =\left(  \lambda-T\right)  X+SX=X$. Thus $\ker\left(  l_{\lambda
}^{0}\right)  =\left\{  0\right\}  $ and $\operatorname{im}\left(  l_{\lambda
}^{1}\right)  =X$. Finally, if $\left(  \zeta,\eta\right)  \in\ker\left(
l_{\lambda}^{1}\right)  $ then $\left(  \lambda-T\right)  \zeta=S\eta$ and
$\eta=\left(  \lambda-C\right)  S^{-1}\zeta$. It follows that $l_{\lambda}%
^{0}S^{-1}\zeta=\left(  SS^{-1}\zeta,\left(  \lambda-C\right)  S^{-1}%
\zeta\right)  =\left(  \zeta,\eta\right)  $, that is, $\ker\left(  l_{\lambda
}^{1}\right)  =\operatorname{im}\left(  l_{\lambda}^{0}\right)  $. Thus
$\mathcal{L}_{\lambda}\left(  T,S\right)  $ is exact for every $\lambda
\in\mathbb{C}$, which means that $\sigma_{l}\left(  T,S\right)  =\varnothing$.
\end{proof}

In the case of a left Banach $\mathfrak{A}_{q}$-module $X$ given by an
operator pair $\left(  T,S\right)  $ with $TS=q^{-1}ST$, we see that $\left(
T,S\right)  $ is a pseudo-commuting pair of operators on $X$ with $C=q^{-1}T$.
In particular, we have its left spectrum $\sigma_{l}\left(  T,S\right)  $.
Thus $\lambda\in\sigma_{l}\left(  T,S\right)  $ iff the following complex
\begin{align*}
\mathcal{L}_{\lambda}\left(  T,S\right)   &  :0\rightarrow
X\overset{l_{\lambda}^{0}}{\longrightarrow}X\oplus X\overset{l_{\lambda}%
^{1}}{\longrightarrow}X\rightarrow0,\\
l_{\lambda}^{0}  &  =\left[
\begin{tabular}
[c]{l}%
$S$\\
$\lambda-q^{-1}T$%
\end{tabular}
\right]  \quad\text{and}\quad l_{\lambda}^{1}=\left[
\begin{tabular}
[c]{ll}%
$\lambda-T$ & $-S$%
\end{tabular}
\right]
\end{align*}
fails to be exact. Actually, the action of $\mathfrak{A}_{q}$ on $X$ provides
many other pseudo-commuting pairs of operators on $X$. Namely, if $f\left(
x\right)  =\sum_{n}a_{n}x^{n}\in\mathfrak{A}_{q}$ is a polynomial then
$f\left(  T\right)  =\sum_{n}a_{n}T^{n}$ and $f\left(  T\right)  S=\sum
_{n}a_{n}T^{n}S=\sum_{n}Sa_{n}q^{-n}T^{n}=Sf\left(  q^{-1}T\right)  $, that
is,
\[
f\left(  T\right)  S=Sf\left(  q^{-1}T\right)  .
\]
It means that $\left(  f\left(  T\right)  ,S\right)  $ is a pseudo-commuting
pair with $C=f\left(  q^{-1}T\right)  $. In a similar way, we have
\[
f\left(  qT\right)  S=Sf\left(  T\right)  ,
\]
which means that $\left(  f\left(  qT\right)  ,S\right)  $ is a
pseudo-commuting pair with $C=f\left(  T\right)  $.

Further, the flipped pair $\left(  S,T\right)  $ is pseudo-commuting too,
namely $ST=T\left(  qS\right)  $ with $C=qS$. By \textit{the right spectrum
}of the pair $\left(  T,S\right)  $, we mean the set
\[
\sigma_{r}\left(  T,S\right)  =\sigma_{l}\left(  S,T\right)  .
\]
By its very definition, $\mu\in\sigma_{r}\left(  T,S\right)  $ iff the
following complex
\begin{align*}
\mathcal{R}_{\mu}\left(  T,S\right)   &  :0\rightarrow X\overset{r_{\mu}%
^{0}}{\longrightarrow}X\oplus X\overset{r_{\mu}^{1}}{\longrightarrow
}X\rightarrow0,\\
r_{\mu}^{0}  &  =\left[
\begin{tabular}
[c]{l}%
$T$\\
$\mu-qS$%
\end{tabular}
\ \right]  \quad\text{and}\quad r_{\mu}^{1}=\left[
\begin{tabular}
[c]{ll}%
$\mu-S$ & $-T$%
\end{tabular}
\ \right]
\end{align*}
is not exact. Notice that $\mathcal{R}_{\mu}\left(  T,S\right)  =\mathcal{L}%
_{\mu}\left(  S,T\right)  $.

As above if $g\left(  y\right)  =\sum_{n}b_{n}y^{n}\in\mathfrak{A}_{q}$ is a
polynomial, then $g\left(  S\right)  T=\sum_{n}b_{n}S^{n}T=\sum_{n}Tb_{n}%
q^{n}S^{n}=Tg\left(  qS\right)  $, and $g\left(  q^{-1}S\right)  T=Tg\left(
S\right)  $, that is, $\left(  g\left(  S\right)  ,T\right)  $ and $\left(
g\left(  q^{-1}S\right)  ,T\right)  $ are pseudo-commuting pairs on $X$. One
can replace $T$ or $S$ by the monomials $T^{k}$ or $S^{k}$. In this case,
$T^{k}S=q^{-k}ST^{k}$ or $TS^{k}=q^{-k}S^{k}T$ and the number $q$ is replaced
by $q^{k}$. They are pairs of the left Banach $\mathfrak{A}_{q^{k}}$-module
$X$.

\subsection{The $q$-projection property, compactness and nonvoidness of the
spectrum}

As above, let $X$ be a left Banach $\mathfrak{A}_{q}$-module given by a pair
$\left(  T,S\right)  $ with $TS=q^{-1}ST$.

\begin{theorem}
\label{lemSpLR}The equalities
\[
\sigma_{x}\left(  T,S\right)  =\sigma_{l}\left(  T,S\right)  \times\left\{
0\right\}  \text{\quad and\quad}\sigma_{y}\left(  T,S\right)  =\left\{
0\right\}  \times\sigma_{r}\left(  T,S\right)
\]
hold. Thus $\sigma\left(  T,S\right)  =\left(  \sigma_{l}\left(  T,S\right)
\times\left\{  0\right\}  \right)  \cup\left(  \left\{  0\right\}
\times\sigma_{r}\left(  T,S\right)  \right)  $ is a union of the compact
subsets and the following $q$-projection property
\[
\sigma\left(  T,S\right)  \subseteq\left(  \left(  \sigma\left(  T\right)
\cup\sigma\left(  q^{-1}T\right)  \right)  \times\left\{  0\right\}  \right)
\cup\left(  \left\{  0\right\}  \times\left(  \sigma\left(  S\right)
\cup\sigma\left(  qS\right)  \right)  \right)
\]
holds.
\end{theorem}

\begin{proof}
Take $\lambda\in\mathbb{C}_{x}$. As we have seen above, the complex
(\ref{Xkos}) is reduced to%
\begin{align*}
\mathcal{K}\left(  \left(  T,S\right)  ,\lambda\right)   &  :0\rightarrow
X\overset{d_{\lambda}^{0}}{\longrightarrow}X\oplus X\overset{d_{\lambda}%
^{1}}{\longrightarrow}X\rightarrow0,\\
d_{\lambda}^{0}  &  =\left[
\begin{tabular}
[c]{l}%
$-qS$\\
$T-q\lambda$%
\end{tabular}
\ \ \right]  \quad\text{and}\quad d_{\lambda}^{1}=\left[
\begin{tabular}
[c]{ll}%
$T-\lambda$ & $S$%
\end{tabular}
\ \ \right]  .
\end{align*}
Notice that $d_{\lambda}^{0}=-ql_{\lambda}^{0}$ and $d_{\lambda}%
^{1}=-l_{\lambda}^{1}$. Then the following commutative diagram%
\[%
\begin{tabular}
[c]{lllllll}%
$\mathcal{K}\left(  \left(  T,S\right)  ,\lambda\right)  :0\rightarrow$ & $X$
& $\overset{d_{\lambda}^{0}=\left[
\begin{tabular}
[c]{l}%
$-qS$\\
$T-q\lambda$%
\end{tabular}
\ \right]  }{\longrightarrow}$ & $X\oplus X$ & $\overset{d_{\lambda}%
^{1}=\left[
\begin{tabular}
[c]{ll}%
$T-\lambda$ & $S$%
\end{tabular}
\ \right]  }{\longrightarrow}$ & $X$ & $\rightarrow0$\\
& $\downarrow q$ &  & $\downarrow-\left(  1\oplus1\right)  $ &  &
$\downarrow1$ & \\
$\mathcal{L}_{\lambda}\left(  T,S\right)  :0\rightarrow$ & $X$ &
$\overset{l_{\lambda}^{0}=\left[
\begin{tabular}
[c]{l}%
$S$\\
$\lambda-q^{-1}T$%
\end{tabular}
\ \ \right]  }{\longrightarrow}$ & $X\oplus X$ & $\overset{l_{\lambda}%
^{1}=\left[
\begin{tabular}
[c]{ll}%
$\lambda-T$ & $-S$%
\end{tabular}
\ \ \right]  }{\longrightarrow}$ & $X$ & $\rightarrow0$%
\end{tabular}
\ \ \ \ \ \
\]
implements an isomorphic identification $\mathcal{K}\left(  \left(
T,S\right)  ,\lambda\right)  =\mathcal{L}_{\lambda}\left(  T,S\right)  $ of
the Banach space complexes. It follows that $\sigma_{x}\left(  T,S\right)
=\sigma_{l}\left(  T,S\right)  \times\left\{  0\right\}  $.

If $\mu\in\mathbb{C}_{y}$ then the complex (\ref{Xkos}) is converted to the
following one
\begin{align*}
\mathcal{K}\left(  \left(  T,S\right)  ,\mu\right)   &  :0\rightarrow
X\overset{d_{\mu}^{0}}{\longrightarrow}X\oplus X\overset{d_{\mu}%
^{1}}{\longrightarrow}X\rightarrow0,\\
d_{\mu}^{0}  &  =\left[
\begin{tabular}
[c]{l}%
$\mu-qS$\\
$T$%
\end{tabular}
\ \ \ \right]  \quad\text{and}\quad d_{\mu}^{1}=\left[
\begin{tabular}
[c]{ll}%
$T$ & $S-\mu$%
\end{tabular}
\ \ \ \right]  .
\end{align*}
As above, the following commutative diagram
\[%
\begin{tabular}
[c]{lllllll}%
$\mathcal{K}\left(  \left(  T,S\right)  ,\mu\right)  :0\rightarrow$ & $X$ &
$\overset{d_{\mu}^{0}=\left[
\begin{tabular}
[c]{l}%
$\mu-qS$\\
$T$%
\end{tabular}
\ \right]  }{\longrightarrow}$ & $X\oplus X$ & $\overset{d_{\mu}^{1}=\left[
\begin{tabular}
[c]{ll}%
$T$ & $S-\mu$%
\end{tabular}
\ \right]  }{\longrightarrow}$ & $X$ & $\rightarrow0$\\
& $\downarrow_{1}$ &  & $\downarrow_{\left[
\begin{tabular}
[c]{ll}%
$0$ & $1$\\
$1$ & $0$%
\end{tabular}
\ \right]  }$ &  & $\downarrow_{-1}$ & \\
$\mathcal{R}_{\mu}\left(  T,S\right)  :0\rightarrow$ & $X$ & $\overset{r_{\mu
}^{0}=\left[
\begin{tabular}
[c]{l}%
$T$\\
$\mu-qS$%
\end{tabular}
\ \right]  }{\longrightarrow}$ & $X\oplus X$ & $\overset{r_{\mu}^{1}=\left[
\begin{tabular}
[c]{ll}%
$\mu-S$ & $-T$%
\end{tabular}
\ \right]  }{\longrightarrow}$ & $X$ & $\rightarrow0$%
\end{tabular}
\ \ \ \
\]
implements an isomorphic identification $\mathcal{K}\left(  \left(
T,S\right)  ,\mu\right)  =\mathcal{R}_{\mu}\left(  S,T\right)  $ of the Banach
space complexes. Hence $\sigma_{y}\left(  T,S\right)  =\left\{  0\right\}
\times\sigma_{r}\left(  T,S\right)  $.

Finally, note that $\mathcal{L}_{\lambda}\left(  T,S\right)  $, $\lambda
\in\mathbb{C}_{x}$, and $\mathcal{R}_{\mu}\left(  T,S\right)  $, $\mu
\in\mathbb{C}_{y}$ are continuously parametrized Banach space complexes over
$\mathbb{C}$. Using \cite[Theorem 2.1]{Tay}, we deduce that $\sigma_{x}\left(
T,S\right)  \subseteq\mathbb{C}_{x}$ and $\sigma_{y}\left(  T,S\right)
\subseteq\mathbb{C}_{y}$ are closed subsets in the standard topology of the
complex plane $\mathbb{C}$. By Lemma \ref{lTSsp}, we obtain that
\begin{align*}
\sigma_{x}\left(  T,S\right)   &  =\sigma_{l}\left(  T,S\right)
\times\left\{  0\right\}  \subseteq\left(  \sigma\left(  T\right)  \cup
\sigma\left(  q^{-1}T\right)  \right)  \times\left\{  0\right\}  ,\\
\sigma_{y}\left(  T,S\right)   &  =\left\{  0\right\}  \times\sigma_{r}\left(
T,S\right)  =\left\{  0\right\}  \times\sigma_{r}\left(  S,T\right)
\subseteq\left\{  0\right\}  \times\left(  \sigma\left(  S\right)  \cup
\sigma\left(  qS\right)  \right)  .
\end{align*}
It follows that
\[
\sigma\left(  T,S\right)  \subseteq\left(  \left(  \sigma\left(  T\right)
\cup\sigma\left(  q^{-1}T\right)  \right)  \times\left\{  0\right\}  \right)
\cup\left(  \left\{  0\right\}  \times\left(  \sigma\left(  S\right)
\cup\sigma\left(  qS\right)  \right)  \right)  ,
\]
that is, $\sigma\left(  T,S\right)  $ is bounded. Consequently, $\sigma
_{x}\left(  T,S\right)  $ and $\sigma_{y}\left(  T,S\right)  $ are compact
subsets. In particular, so is $\sigma\left(  T,S\right)  $ in $\mathbb{C}^{2}$.
\end{proof}

The nonvoidness of the joint spectrum follows from the following assertion.

\begin{proposition}
\label{propNVs}Let $X$ be a left Banach $\mathfrak{A}_{q}$-module given by an
operator pair $\left(  T,S\right)  $ with $TS=q^{-1}ST$. If $T$ and $S$ are
invertible operators, then $\sigma\left(  T,S\right)  =\varnothing$ and
$\left\vert q\right\vert =1$. If $\left\vert q\right\vert \neq1$ then
$\sigma\left(  T,S\right)  $ is a nonempty compact subset in $\mathbb{C}^{2}$.
\end{proposition}

\begin{proof}
First notice that $\sigma\left(  TS\right)  \cup\left\{  0\right\}
=\sigma\left(  ST\right)  \cup\left\{  0\right\}  =q\sigma\left(  TS\right)
\cup\left\{  0\right\}  $, therefore $\sigma\left(  TS\right)  =\left\{
0\right\}  $ whenever $\left\vert q\right\vert \neq1$. It follows that $TS$ is
not invertible, therefore so is $T$ or $S$. Thus, if $T$ and $S$ are
invertible operators, then $\left\vert q\right\vert =1$. Moreover, $\sigma
_{l}\left(  T,S\right)  =\varnothing$ and $\sigma_{r}\left(  T,S\right)
=\sigma_{l}\left(  S,T\right)  =\varnothing$ thanks to Lemma \ref{lTSsp}.
Using Theorem \ref{lemSpLR}, we conclude that $\sigma\left(  T,S\right)
=\left(  \sigma_{l}\left(  T,S\right)  \times\left\{  0\right\}  \right)
\cup\left(  \left\{  0\right\}  \times\sigma_{r}\left(  T,S\right)  \right)
=\varnothing$.

Now assume that $\left\vert q\right\vert \neq1$. Then $0\in\sigma\left(
T\right)  \cup\sigma\left(  S\right)  $. First assume that $0\in\sigma\left(
T\right)  $. If $Y=\ker\left(  T\right)  \neq\left\{  0\right\}  $ then
$S\left(  Y\right)  \subseteq Y$, for $TS=q^{-1}ST$. The approximate point
spectrum $\sigma_{\operatorname{ap}}\left(  qS|Y\right)  $ of the operator
$qS|Y\in\mathcal{B}\left(  Y\right)  $ is non-empty. Take $\mu\in
\sigma_{\operatorname{ap}}\left(  qS|Y\right)  $. Then $\left(  \mu-qS\right)
\zeta_{n}\rightarrow0$, $n\rightarrow\infty$ for a certain sequence $\left\{
\zeta_{n}\right\}  $ of unit vectors in $Y$. It follows that
\[
r_{\mu}^{0}\left(  \zeta_{n}\right)  =\left[
\begin{tabular}
[c]{l}%
$T$\\
$\mu-qS$%
\end{tabular}
\ \ \right]  \zeta_{n}=\left(  T\zeta_{n},\left(  \mu-qS\right)  \zeta
_{n}\right)  =\left(  0,\left(  \mu-qS\right)  \zeta_{n}\right)
\rightarrow0\text{,}%
\]
where $r_{\mu}^{0}$ is the differential of the complex $\mathcal{R}_{\mu
}\left(  T,S\right)  $. Thus either $\ker\left(  r_{\mu}^{0}\right)
\neq\left\{  0\right\}  $ or $\operatorname{im}\left(  r_{\mu}^{0}\right)  $
is not closed. Consequently, $\mu\in\sigma_{r}\left(  T,S\right)  $, $\left(
0,\mu\right)  \in\sigma_{y}\left(  T,S\right)  $ and $\sigma_{y}\left(
T,S\right)  \subseteq\sigma\left(  T,S\right)  $ by virtue of Theorem
\ref{lemSpLR}. If $\ker\left(  T\right)  =\left\{  0\right\}  $ then $TX\neq
X$. Consider the following commutative diagram
\[%
\begin{tabular}
[c]{lll}%
$X$ & $\overset{T}{\longrightarrow}$ & $X$\\
$\downarrow^{qS}$ &  & $\downarrow^{S}$\\
$X$ & $\overset{T}{\longrightarrow}$ & $X$%
\end{tabular}
\
\]
($ST=qTS$). By Slodkowski Lemma \cite[Lemma 1.6]{Muller}, we have $TX+\left(
\mu-S\right)  X\neq X$ for a certain $\mu\in\mathbb{C}$. In particular,
$\mu\in\sigma\left(  S\right)  $ and $\operatorname{im}\left(  \left[
\begin{tabular}
[c]{ll}%
$\mu-S$ & $-T$%
\end{tabular}
\ \right]  \right)  \neq X$, that is, $\operatorname{im}\left(  r_{\mu}%
^{1}\right)  \neq X$. Consequently, $\left(  0,\mu\right)  \in\sigma
_{y}\left(  T,S\right)  $ by Theorem \ref{lemSpLR}.

A similar argument is applicable to the case of $0\in\sigma\left(  S\right)
$. Namely, if $\ker\left(  S\right)  \neq\left\{  0\right\}  $ then as above
$\operatorname{im}\left(  l_{\lambda}^{0}\right)  $ is not closed for a
certain $\lambda\in\sigma_{\operatorname{ap}}\left(  q^{-1}T|\ker\left(
S\right)  \right)  $, where $l_{\lambda}^{0}$ is the differential of the
complex $\mathcal{L}_{\lambda}\left(  T,S\right)  $. Therefore $\left(
\lambda,0\right)  \in\sigma_{x}\left(  T,S\right)  $ and $\sigma_{x}\left(
T,S\right)  \subseteq\sigma\left(  T,S\right)  $ by Theorem \ref{lemSpLR}. If
$SX\neq X$ then we consider the following commutative diagram
\[%
\begin{tabular}
[c]{lll}%
$X$ & $\overset{S}{\longrightarrow}$ & $X$\\
$\downarrow^{q^{-1}T}$ &  & $\downarrow^{T}$\\
$X$ & $\overset{S}{\longrightarrow}$ & $X$%
\end{tabular}
\ \ \
\]
($TS=q^{-1}ST$). Again by Slodkowski Lemma \cite[Lemma 1.6]{Muller}, we deduce
that $\left(  \lambda-T\right)  X+SX\neq X$ for a certain $\lambda\in
\sigma\left(  T\right)  $. It follows that $\operatorname{im}\left(
l_{\lambda}^{1}\right)  \neq X$, that is, $\left(  \lambda,0\right)  \in
\sigma_{x}\left(  T,S\right)  $. Whence $\sigma\left(  T,S\right)
\neq\varnothing$ whenever $\left\vert q\right\vert \neq1$.
\end{proof}

\begin{remark}
As we can see from the proof of Proposition \ref{propNVs} that $\sigma
_{y}\left(  T,S\right)  \neq\varnothing$ whenever $0\in\sigma\left(  T\right)
$, whereas $0\in\sigma\left(  S\right)  $ implies that $\sigma_{x}\left(
T,S\right)  \neq\varnothing$.
\end{remark}

\section{The concrete example and lack of the projection property}

Let $X$ be a left Banach $\mathfrak{A}_{q}$-module given by an operator pair
$\left(  T,S\right)  $ with $TS=q^{-1}ST$ and $\left\vert q\right\vert <1$. By
Proposition \ref{propNVs}, the joint spectrum $\sigma\left(  T,S\right)  $ is
a nonempty compact subset in $\mathbb{C}^{2}$. By Theorem \ref{lemSpLR}, the
inclusion $\sigma\left(  T,S\right)  |\mathbb{C}_{x}\subseteq\left(
\sigma\left(  T\right)  \cup\sigma\left(  q^{-1}T\right)  \right)
\times\left\{  0\right\}  $ (or $\sigma_{l}\left(  T,S\right)  \subseteq
\sigma\left(  T\right)  \cup\sigma\left(  q^{-1}T\right)  $) holds. The
forward projection property for the joint spectrum would state the inclusion
$\sigma\left(  T,S\right)  |\mathbb{C}_{x}\subseteq\sigma\left(  T\right)
\times\left\{  0\right\}  $ rather than what we have seen above. In a similar
way $\left\{  0\right\}  \times\sigma\left(  S\right)  \subseteq\sigma\left(
T,S\right)  |\mathbb{C}_{y}$ would be our expectation for the backward
projection property. Our present task is to show that none of these inclusions
valid for the joint spectrum of a contractive Banach quantum plane.

\subsection{The left Banach $\mathfrak{A}_{q}$-modules $\ell_{p}$}

There are key examples of the operator $q$-planes considered on the $L_{p}%
$-Banach spaces. To be certain, we assume that $X=\ell_{p}$ (or $\ell
_{p}\left(  \mathbb{Z}_{+}\right)  $) for $1\leq p<\infty$ with its standard
topological basis $\left\{  e_{n}\right\}  _{n\in\mathbb{Z}_{+}}$, and
consider the unilateral shift operator $T$ and the diagonal $q$-operator $S$
on $\ell_{p}$. Thus $T,S\in\mathcal{B}\left(  \ell_{p}\right)  $ are the
operators acting by the rules%
\begin{align*}
T\left(  e_{n}\right)   &  =e_{n+1},\\
S\left(  e_{n}\right)   &  =q^{n}e_{n},
\end{align*}
for all $n\in\mathbb{Z}_{+}$. One can easily verify that $TS=q^{-1}ST$. Since
$\left\vert q\right\vert <1$, it follows that $S$ is a compact operator with
its spectrum $\sigma\left(  S\right)  =\left\{  1\right\}  _{q}=\left\{
q^{n}:n\in\mathbb{Z}_{+}\right\}  \cup\left\{  0\right\}  $, which is the
$q$-hull of $\left\{  1\right\}  $. Moreover, $\sigma\left(  T\right)
=\mathbb{D}_{1}$ is the closed unit disk. Thus both $\sigma\left(  T\right)  $
and $\sigma\left(  S\right)  $ are the $q$-compact (coincide with their
$q$-hulls) subsets of the complex plane.

\subsection{The left spectrum}

First we focus on the left spectrum $\sigma_{l}\left(  T,S\right)  $ or
$\sigma_{x}\left(  T,S\right)  $ to be described. Since $0\in\sigma\left(
S\right)  $, it follows as in the proof of Proposition \ref{propNVs} that
$\left(  \lambda-T\right)  X+SX\neq X$ for a certain $\lambda\in\sigma\left(
T\right)  $ (the Slodkowski Lemma). Hence the second cohomology group
$H^{2}\left(  \lambda\right)  =\ell_{p}/\operatorname{im}\left(  l_{\lambda
}^{1}\right)  $ of the complex $\mathcal{L}_{\lambda}\left(  T,S\right)  $ is
not trivial, therefore $\left(  \lambda,0\right)  \in\sigma\left(  T,S\right)
$ or $\lambda\in\sigma_{l}\left(  T,S\right)  $. Hence $\left(  \sigma\left(
T\right)  \times\left\{  0\right\}  \right)  \cap\left(  \sigma\left(
T,S\right)  |\mathbb{C}_{x}\right)  \neq\varnothing$ or $\sigma\left(
T\right)  \cap\sigma_{l}\left(  T,S\right)  \neq\varnothing$. Since
$\ker\left(  S\right)  =\left\{  0\right\}  $, we conclude that $H^{0}\left(
\lambda\right)  =\ker\left(  l_{\lambda}^{0}\right)  =\ker\left(  S\right)
\cap\ker\left(  \lambda-q^{-1}T\right)  =\left\{  0\right\}  $ for all
$\lambda\in\sigma_{l}\left(  T,S\right)  $.

\begin{proposition}
\label{propEx1}If $1<\left\vert \lambda\right\vert \leq\left\vert q\right\vert
^{-1}$ then $H^{0}\left(  \lambda\right)  =H^{2}\left(  \lambda\right)
=\left\{  0\right\}  $ and $H^{1}\left(  \lambda\right)  \neq\left\{
0\right\}  $. Thus $\mathbb{D}_{\left\vert q\right\vert ^{-1}}\backslash
\mathbb{D}_{1}\subseteq\sigma_{l}\left(  T,S\right)  \subseteq\mathbb{D}%
_{\left\vert q\right\vert ^{-1}}$ and $\left(  0,0\right)  \notin\sigma\left(
T,S\right)  $.
\end{proposition}

\begin{proof}
Take $\lambda\in\mathbb{D}_{\left\vert q\right\vert ^{-1}}\backslash
\mathbb{D}_{1}$, and consider the Banach space complex $\mathcal{L}_{\lambda
}\left(  T,S\right)  :$
\[%
\begin{tabular}
[c]{lllllll}%
$0\rightarrow$ & $\ell_{p}$ & $\overset{l_{\lambda}^{0}=\left[
\begin{tabular}
[c]{l}%
$S$\\
$\lambda-q^{-1}T$%
\end{tabular}
\ \right]  }{\longrightarrow}$ & $\ell_{p}\oplus\ell_{p}$ &
$\overset{l_{\lambda}^{1}=\left[
\begin{tabular}
[c]{ll}%
$\lambda-T$ & $-S$%
\end{tabular}
\ \right]  }{\longrightarrow}$ & $\ell_{p}$ & $\rightarrow0.$%
\end{tabular}
\ \
\]
Since $\lambda\notin\sigma\left(  T\right)  $, it follows that
$\operatorname{im}\left(  l_{\lambda}^{1}\right)  =\operatorname{im}\left(
\lambda-T\right)  +\operatorname{im}\left(  S\right)  =\ell_{p}$, that is,
$H^{2}\left(  \lambda\right)  =\left\{  0\right\}  $.

To prove that $H^{1}\left(  \lambda\right)  \neq\left\{  0\right\}  $, let us
take the following vectors $\zeta=\sum_{n=0}^{\infty}\dfrac{1}{\lambda^{n+1}%
}e_{n}$ with $\left\Vert \zeta\right\Vert _{p}=\left(  \left\vert
\lambda\right\vert ^{p}-1\right)  ^{-1/p}$, and $\eta=e_{0}$ with $\left\Vert
\eta\right\Vert _{p}=1$ from the Banach space $\ell_{p}$. First prove that
$\left(  \zeta,\eta\right)  \notin\operatorname{im}\left(  l_{\lambda}%
^{0}\right)  $. Indeed, in the contrary case, we have $\left(  \zeta
,\eta\right)  =l_{\lambda}^{0}\left(  \theta\right)  $ for a certain
$\theta=\sum_{n=0}^{\infty}\alpha_{n}e_{n}\in\ell_{p}$, which means that
$\zeta=S\left(  \theta\right)  $ and $\left(  \lambda-q^{-1}T\right)  \left(
\theta\right)  =\eta$. It follows that $q^{n}\alpha_{n}=\dfrac{1}%
{\lambda^{n+1}}$ for all $n$. Then
\[
\lim\inf_{n}\left\{  \left\vert \alpha_{n}\right\vert \right\}  =\dfrac
{1}{\left\vert \lambda\right\vert }\lim\inf_{n}\left\{  \dfrac{1}{\left\vert
q\lambda\right\vert ^{n}}\right\}  \geq\dfrac{1}{\left\vert \lambda\right\vert
}>0,
\]
a contradiction ($\theta\in\ell_{p}$ with $\left\Vert \theta\right\Vert
_{p}^{p}=\sum_{n}\left\vert \alpha_{n}\right\vert ^{p}<\infty$). Notice also
that
\[
\eta=\left(  \lambda-q^{-1}T\right)  \left(  \theta\right)  =\lambda\alpha
_{0}e_{0}+\sum_{n=1}^{\infty}\left(  \lambda\alpha_{n}-q^{-1}\alpha
_{n-1}\right)  e_{n},
\]
which results in the same equalities $\alpha_{0}=\dfrac{1}{\lambda}$ and
$\alpha_{n}=\dfrac{1}{q^{n}\lambda^{n+1}}$, $n\in\mathbb{N}$. Further, let us
prove that $\left(  \zeta,\eta\right)  \in\ker\left(  l_{\lambda}^{1}\right)
$. Namely,
\[
l_{\lambda}^{1}\left(  \zeta,\eta\right)  =\left(  \lambda-T\right)
\zeta-S\eta=\sum_{n=0}^{\infty}\dfrac{1}{\lambda^{n}}e_{n}-\sum_{n=0}^{\infty
}\dfrac{1}{\lambda^{n+1}}e_{n+1}-e_{0}=0.
\]
Hence $H^{1}\left(  \lambda\right)  \neq\left\{  0\right\}  $. Thus
$\mathbb{D}_{\left\vert q\right\vert ^{-1}}\backslash\mathbb{D}_{1}%
\subseteq\sigma_{l}\left(  T,S\right)  $. By Theorem \ref{lemSpLR}, we derive
that
\[
\sigma_{l}\left(  T,S\right)  \subseteq\sigma\left(  T\right)  \cup
q^{-1}\sigma\left(  T\right)  =\mathbb{D}_{1}\cup q^{-1}\mathbb{D}%
_{1}=\mathbb{D}_{\left\vert q\right\vert ^{-1}}.
\]
Thus we obtain the inclusions $\mathbb{D}_{\left\vert q\right\vert ^{-1}%
}\backslash\mathbb{D}_{1}\subseteq\sigma_{l}\left(  T,S\right)  \subseteq
\mathbb{D}_{\left\vert q\right\vert ^{-1}}$.

Now prove that $\left(  0,0\right)  \notin\sigma\left(  T,S\right)  $, that
is, the complex (see (\ref{Xkos})) $\mathcal{K}\left(  \left(  T,S\right)
,0\right)  :$
\begin{align*}
0  &  \rightarrow\ell_{p}\overset{d^{0}}{\longrightarrow}\ell_{p}\oplus
\ell_{p}\overset{d^{1}}{\longrightarrow}\ell_{p}\rightarrow0\\
d^{0}  &  =\left[
\begin{tabular}
[c]{l}%
$-qS$\\
$T$%
\end{tabular}
\ \right]  ,\quad d^{1}=\left[
\begin{tabular}
[c]{ll}%
$T$ & $S$%
\end{tabular}
\ \right]
\end{align*}
is exact. Since $\ker(S)=\left\{  0\right\}  $, it follows that $\ker\left(
d^{0}\right)  =\left\{  0\right\}  $. For every $\zeta=\sum_{n=0}^{\infty
}\alpha_{n}e_{n}\in\ell_{p}$ we have $\zeta=S\left(  \alpha_{0}e_{0}\right)
+T\left(  \eta\right)  =d^{1}\left(  \eta,\alpha_{0}e_{0}\right)  $ with
$\eta=\sum_{n=0}^{\infty}\alpha_{n+1}e_{n}$, which means that
$\operatorname{im}\left(  d^{1}\right)  =\ell_{2}$. Finally, if $\left(
\zeta,\eta\right)  \in\ker\left(  d^{1}\right)  $ with $\zeta=\sum
_{n=0}^{\infty}\alpha_{n}e_{n}$ and $\ \eta=\sum_{n=0}^{\infty}\beta_{n}e_{n}%
$, then $T\zeta=-S\eta$, that is, $\beta_{0}=0$ and $\alpha_{n}=-q^{n+1}%
\beta_{n+1}$ for all $n\in\mathbb{Z}_{+}$. It remains to notice that
$d^{0}\theta=\left(  -qS\theta,T\theta\right)  =\left(  \zeta,\eta\right)  $
for $\theta=\sum_{n=0}^{\infty}\beta_{n+1}e_{n}\in\ell_{p}$. Whence
$\ker\left(  d^{1}\right)  =\operatorname{im}\left(  d^{0}\right)  $, thereby
the complex is exact.
\end{proof}

\begin{corollary}
\label{corRSQ}The inclusion $\sigma_{r}\left(  T,S\right)  \subseteq\left\{
q^{m},q^{m-1},\ldots,q,1\right\}  $ holds for some $m\geq0$.
\end{corollary}

\begin{proof}
By Proposition \ref{propEx1}, we have $\left(  0,0\right)  \notin\sigma\left(
T,S\right)  $. In particular, $0\notin\sigma_{r}\left(  T,S\right)  $. By
Proposition \ref{propNVs}, we deduce that $\sigma_{r}\left(  T,S\right)  \cap
B\left(  0,\varepsilon\right)  =\varnothing$ for some small $\varepsilon>0$.
Using Theorem \ref{lemSpLR}, the inclusions $\sigma_{r}\left(  S,T\right)
\subseteq\sigma\left(  S\right)  \cup\sigma\left(  qS\right)  \subseteq
\left\{  1\right\}  _{q}$ hold. It follows that
\[
\sigma_{r}\left(  S,T\right)  \subseteq\left\{  1\right\}  _{q}\cap\left(
\mathbb{C}\backslash B\left(  0,\varepsilon\right)  \right)  =\left\{
q^{m},q^{m-1},\ldots,q,1\right\}
\]
for some $m\geq0$.
\end{proof}

As follows from Proposition \ref{propEx1}, $\sigma_{l}\left(  T,S\right)  \cap
B\left(  0,\varepsilon\right)  =\varnothing$ holds too for a small disk
$B\left(  0,\varepsilon\right)  $, which is a $q$-open (coincides with its
$q$-hull) subset. The presence of a $q$-open subset $U\subseteq\mathbb{C}%
\backslash\sigma_{l}\left(  T,S\right)  $ is used to relate to the
transversality property of the left Banach $\mathfrak{A}_{q}$-module $\ell
_{p}$ (see \cite{DMedJM09}).

\subsection{The right spectrum}

The following assertion clarifies the statement of Corollary \ref{corRSQ}.

\begin{proposition}
\label{propEx2}The equality $\sigma_{r}\left(  T,S\right)  =\left\{
1\right\}  $ holds. In this case, $H^{0}\left(  1\right)  =\left\{  0\right\}
$ and $H^{1}\left(  1\right)  =H^{2}\left(  1\right)  =\mathbb{C}$, where
$H^{i}\left(  1\right)  $ are the cohomology groups of the complex
$\mathcal{R}_{1}\left(  T,S\right)  $. Thus $\mathcal{R}_{1}\left(
T,S\right)  $ is a nonexact Fredholm complex.
\end{proposition}

\begin{proof}
Based on Theorem \ref{lemSpLR}, we consider the following parametrized Banach
space complex
\[%
\begin{tabular}
[c]{lllllll}%
$\mathcal{R}_{\mu}\left(  T,S\right)  :0\rightarrow$ & $\ell_{p}$ &
$\overset{r_{\mu}^{0}=\left[
\begin{tabular}
[c]{l}%
$T$\\
$\mu-qS$%
\end{tabular}
\ \right]  }{\longrightarrow}$ & $\ell_{p}\oplus\ell_{p}$ & $\overset{r_{\mu
}^{1}=\left[
\begin{tabular}
[c]{ll}%
$\mu-S$ & $-T$%
\end{tabular}
\ \right]  }{\longrightarrow}$ & $\ell_{p}$ & $\rightarrow0$%
\end{tabular}
\ \ \ \
\]
whose cohomology groups are denoted by $H^{i}\left(  \mu\right)  $, $i=0,1,2$.
Since $\ker\left(  T\right)  =\left\{  0\right\}  $, it follows that
$H^{0}\left(  \mu\right)  =\left\{  0\right\}  $ for all $\mu\in\mathbb{C}$.
If $\mu=1$ then $\left(  \mu-S\right)  \left(  e_{0}\right)  =0$ and
$\operatorname{im}\left(  \mu-S\right)  \subseteq\operatorname{im}\left(
T\right)  $, which means that $\operatorname{im}\left(  r_{\mu}^{1}\right)
=\operatorname{im}\left(  \mu-S\right)  +\operatorname{im}\left(  T\right)
=\operatorname{im}\left(  T\right)  \neq\ell_{p}$, that is, $H^{2}\left(
1\right)  =\ell_{p}/\operatorname{im}\left(  T\right)  =\mathbb{C}$.

Further, note that $\left(  e_{0},0\right)  \in\ker\left(  r_{1}^{1}\right)
$, and for every $\left(  \zeta,\eta\right)  \in\ker\left(  r_{1}^{1}\right)
$ with $\zeta=\sum_{n}\alpha_{n}e_{n}$, $\eta=\sum_{n}\beta_{n}e_{n}\in
\ell_{p}$, we have $\left(  \zeta,\eta\right)  =\alpha_{0}\left(
e_{0},0\right)  +\left(  \zeta^{\prime},\eta\right)  $ with $\zeta^{\prime
}=\sum_{n\geq1}\alpha_{n}e_{n}$.. Since $e_{0}\notin\operatorname{im}\left(
T\right)  $, it follows that $\left(  e_{0},0\right)  \notin\operatorname{im}%
\left(  r_{1}^{0}\right)  $. Thus we can assume that $\alpha_{0}=0$. In this
case, $\left(  1-q^{n}\right)  \alpha_{n}=\beta_{n-1}$ for all $n\geq1$. Put
$\gamma_{n}=\alpha_{n+1}$, $n\geq0$, and $\theta=\sum_{n}\gamma_{n}e_{n}%
\in\ell_{p}$. Then
\begin{align*}
r_{1}^{0}\left(  \theta\right)   &  =\left(  T\theta,\left(  1-qS\right)
\theta\right)  =\left(  \sum_{n\geq1}\gamma_{n-1}e_{n},\sum_{n}\left(
1-q^{n+1}\right)  \gamma_{n}e_{n}\right) \\
&  =\left(  \sum_{n\geq1}\alpha_{n}e_{n},\sum_{n}\left(  1-q^{n+1}\right)
\alpha_{n+1}e_{n}\right)  =\left(  \zeta,\eta\right)  .
\end{align*}
Thus $H^{1}\left(  1\right)  =\mathbb{C}$ too. Notice also that $q^{-1}%
\notin\sigma\left(  S\right)  $, which means that $1-qS$ is invertible and
$\operatorname{im}\left(  r_{1}^{0}\right)  $ is closed.

Now based on Corollary \ref{corRSQ}, let us assume that $\mu=q^{k}$ with
$k>0$. Then $\left(  \mu-S\right)  \left(  e_{0}\right)  =\left(
\mu-1\right)  e_{0}\neq0$ and $\operatorname{im}\left(  r_{\mu}^{1}\right)
=\operatorname{im}\left(  \mu-S\right)  +\operatorname{im}\left(  T\right)
=\ell_{p}$, that is, $H^{2}\left(  \mu\right)  =\left\{  0\right\}  $. Take
$\zeta=\sum_{n}\alpha_{n}e_{n}$, $\eta=\sum_{n}\beta_{n}e_{n}\in\ell_{p}$ with
$\left(  \zeta,\eta\right)  \in\ker\left(  r_{\mu}^{1}\right)  $. Then
$\left(  \mu-S\right)  \zeta=T\eta$, which means that $\left(  \mu-1\right)
\alpha_{0}=0$ and $\left(  \mu-q^{n}\right)  \alpha_{n}=\beta_{n-1}$ for all
$n\geq1$. It follows that $\alpha_{0}=0$, $\alpha_{n}=\beta_{n-1}/\left(
\mu-q^{n}\right)  $ for all $n\neq k$, $\beta_{k-1}=0$, and $\alpha_{k}$
remains free. Put $\theta=\sum_{n}\gamma_{n}e_{n}$ with $\gamma_{n}=\beta
_{n}/\left(  \mu-q^{n+1}\right)  $, $n\neq k-1$ and $\gamma_{k-1}=\alpha_{k}$.
Since $\left\vert \gamma_{n}\right\vert ^{p}=\left\vert \beta_{n}\right\vert
^{p}\left\vert \mu-q^{n+1}\right\vert ^{-p}\leq2^{p}\left\vert \beta
_{n}\right\vert ^{p}\left\vert q\right\vert ^{-kp}$ for all large $n$, it
follows that $\left\Vert \theta\right\Vert _{p}\leq c\left\Vert \eta
\right\Vert _{p}<\infty$ for some positive $c$, which depends on $\mu$. Thus
$\theta\in\ell_{p}$ and
\[
r_{\mu}^{0}\left(  \theta\right)  =\left(  T\theta,\left(  \mu-qS\right)
\theta\right)  =\left(  \sum_{n\geq1}\gamma_{n-1}e_{n},\sum_{n\geq0}\left(
\mu-q^{n+1}\right)  \gamma_{n}e_{n}\right)  =\left(  \sum_{n}\alpha_{n}%
e_{n},\sum_{n}\beta_{n}e_{n}\right)  =\left(  \zeta,\eta\right)  ,
\]
that is, $\left(  \zeta,\eta\right)  \in\operatorname{im}\left(  r_{\mu}%
^{0}\right)  $. Hence $H^{1}\left(  \mu\right)  =\left\{  0\right\}  $, and
the equality $\sigma_{r}\left(  T,S\right)  =\left\{  1\right\}  $ follows.

Thus $\sigma_{y}^{\pi,0}\left(  T,S\right)  =\varnothing$, and $\sigma
_{y}^{\pi,m}\left(  T,S\right)  =\sigma_{y}^{\delta,n}\left(  T,S\right)
=\sigma_{y}\left(  T,S\right)  =\left\{  \left(  0,1\right)  \right\}  $ for
all $m\geq1$, and for all $n$.
\end{proof}

\subsection{The first cohomology groups over the interior of the unit disk}

The following assertion states that for the interior of the disk
$\mathbb{D}_{1}$ the spectral points from $\sigma_{l}\left(  T,S\right)  $
could only occur in the first cohomology groups too.

\begin{lemma}
\label{lemH0H2}If $\left\vert \lambda\right\vert <1$ then $H^{0}\left(
\lambda\right)  =H^{2}\left(  \lambda\right)  =\left\{  0\right\}  $, where
$H^{i}\left(  \lambda\right)  $ are the cohomology groups of the complex
$\mathcal{L}_{\lambda}\left(  T,S\right)  $.
\end{lemma}

\begin{proof}
As we have stated above $H^{0}\left(  \lambda\right)  =\left\{  0\right\}  $
out of the fact $\ker\left(  S\right)  =\left\{  0\right\}  $. Since
$\left\vert \lambda\right\vert <1$, it follows that $\ker\left(
\lambda-T\right)  =\left\{  0\right\}  $, $\operatorname{im}\left(
\lambda-T\right)  $ is closed, and $\dim\left(  \ell_{p}/\operatorname{im}%
\left(  \lambda-T\right)  \right)  =1$ (see \cite[VII, 6.5]{Con}). Actually,
the absolute polar $\operatorname{im}\left(  \lambda-T\right)  ^{\perp}%
=\ker\left(  \lambda^{\ast}-T^{\ast}\right)  =\left\langle \zeta_{\lambda
}\right\rangle $ is the subspace generated by the vector $\zeta_{\lambda}%
=\sum_{n}\left(  \lambda^{\ast}\right)  ^{n}e_{n}\in\ell_{q}$ with
$p^{-1}+q^{-1}=1$. If $H^{2}\left(  \lambda\right)  \neq\left\{  0\right\}  $
then $\operatorname{im}\left(  l_{\lambda}^{1}\right)  =\operatorname{im}%
\left(  \lambda-T\right)  +\operatorname{im}\left(  S\right)  \neq\ell_{p}$,
which means that $\operatorname{im}\left(  S\right)  \subseteq
\operatorname{im}\left(  \lambda-T\right)  $. Using Phillips Theorem
\cite[7.4.13]{Kut}, we conclude that $S=\left(  \lambda-T\right)  A$ for some
operator $A\in\mathcal{B}\left(  \ell_{p}\right)  $. Then $S^{\ast}=A^{\ast
}\left(  \lambda^{\ast}-T^{\ast}\right)  $ in $\mathcal{B}\left(  \ell
_{q}\right)  $, and $S^{\ast}\zeta_{\lambda}=0$, that is, $\zeta_{\lambda}%
\in\ker\left(  S^{\ast}\right)  $. But $\ker\left(  S^{\ast}\right)
=\operatorname{im}\left(  S\right)  ^{\perp}=\left\{  0\right\}  $, a contradiction.
\end{proof}

\begin{remark}
If $\left\vert \lambda\right\vert <\left\vert q\right\vert $ then
$\eta_{\lambda}=\sum_{n}\left(  \lambda/q\right)  ^{n}e_{n}\in\ell_{p}$ and
$\left\langle S\eta_{\lambda},\zeta_{\lambda}\right\rangle =\sum_{n\geq
0}\left\vert \lambda\right\vert ^{n}\geq1$, where $\left\langle \cdot
,\cdot\right\rangle $ indicates to the canonical duality $\ell_{p}\times
\ell_{q}\rightarrow\mathbb{C}$. Hence $\operatorname{im}\left(  S\right)
\nsubseteqq\operatorname{im}\left(  \lambda-T\right)  $ or $H^{2}\left(
\lambda\right)  =\left\{  0\right\}  $.
\end{remark}

The presence of a nonzero $H^{1}\left(  \lambda\right)  $ for $0<\left\vert
\lambda\right\vert <1$ (see Proposition \ref{propEx1}) remains unclear. To be
precise, let us consider the case of $p=2$. For every $\eta=\sum_{n}\beta
_{n}e_{n}\in\ell_{2\text{ }}$ we denote by $\eta_{n}$ its $n$th Fourier
tail-sum $\eta_{n}=\sum_{k=n+1}^{\infty}\beta_{k}e_{n}$. As above in Lemma
\ref{lemH0H2}, we use the notation $\zeta_{\lambda q}=\sum_{n}\left(  q^{\ast
}\lambda^{\ast}\right)  ^{n}e_{n}$ to be a vector from $\ell_{2}$, and
consider its orthogonal complement $\left\{  \zeta_{\lambda q}\right\}
^{\perp}\subseteq\ell_{2\text{ }}$. Put
\[
M_{\lambda}=\left\{  \eta\in\left\{  \zeta_{\lambda q}\right\}  ^{\perp
}:\epsilon_{\lambda}\left(  \eta\right)  =\sum_{n\geq0}\left\vert \left(
\eta_{n},T^{n+1}\zeta_{\lambda q}\right)  \right\vert ^{2}<\infty\right\}
\]
to be a subset of $\left\{  \zeta_{\lambda q}\right\}  ^{\perp}$. If
$\eta,\theta\in M_{\lambda}$ then
\begin{align*}
\epsilon_{\lambda}\left(  \eta+\theta\right)   &  \leq\epsilon_{\lambda
}\left(  \eta\right)  +\epsilon_{\lambda}\left(  \theta\right)  +2\sum
_{n\geq0}\left\vert \left(  \eta_{n},T^{n+1}\zeta_{\lambda q}\right)
\right\vert \left\vert \left(  \theta_{n},T^{n+1}\zeta_{\lambda q}\right)
\right\vert \\
&  \leq\epsilon_{\lambda}\left(  \eta\right)  +\epsilon_{\lambda}\left(
\theta\right)  +2\epsilon_{\lambda}\left(  \eta\right)  ^{1/2}\epsilon
_{\lambda}\left(  \theta\right)  ^{1/2}=\left(  \epsilon_{\lambda}\left(
\eta\right)  ^{1/2}+\epsilon_{\lambda}\left(  \theta\right)  ^{1/2}\right)
^{2}<\infty.
\end{align*}
Thereby $M_{\lambda}$ is a subspace of $\left\{  \zeta_{\lambda q}\right\}
^{\perp}$. For every $z\in\mathbb{C}$ with $0\leq\left\vert z\right\vert <1$,
the vector $\eta=\sum_{n}\beta_{n}e_{n}$ with $\beta_{0}=-zq\lambda/\left(
1-zq\lambda\right)  $ and $\beta_{n}=z^{n}$, $n\geq1$ belongs to $M_{\lambda}%
$. Indeed, $\left(  \eta,\zeta_{\lambda q}\right)  =\beta_{0}+\sum_{n\geq
1}\left(  zq\lambda\right)  ^{n}=0$ and
\[
\epsilon_{\lambda}\left(  \eta\right)  =\sum_{n\geq0}\left\vert \sum
_{k=n+1}^{\infty}z^{k}\left(  q\lambda\right)  ^{k-n-1}\right\vert ^{2}%
=\sum_{n\geq0}\left\vert z\right\vert ^{2n+2}\left\vert \frac{1}{1-zq\lambda
}\right\vert ^{2}=\left\vert \frac{z}{1-zq\lambda}\right\vert ^{2}\frac
{1}{1-\left\vert z\right\vert ^{2}}<\infty,
\]
that is, $\eta\in M_{\lambda}$. Actually, every vector $\eta\in\left\{
\zeta_{\lambda q}\right\}  ^{\perp}$ with $\sum_{n\geq0}\left\Vert \eta
_{n}\right\Vert ^{2}<\infty$ (or $\sum_{n\geq0}\sum_{k>n}\left\vert \beta
_{k}\right\vert ^{2}<\infty$) belongs to the subspace $M_{\lambda}$.

\begin{lemma}
\label{lemH1}If $0<\left\vert \lambda\right\vert <1$ then there is a canonical
topological isomorphic identification
\[
H^{1}\left(  \lambda\right)  =\left\{  \zeta_{\lambda q}\right\}  ^{\perp
}/M_{\lambda},
\]
where $H^{1}\left(  \lambda\right)  $ is the first cohomology group of the
complex $\mathcal{L}_{\lambda}\left(  T,S\right)  $.
\end{lemma}

\begin{proof}
Take $\left(  \zeta,\eta\right)  \in\ker\left(  l_{\lambda}^{1}\right)  $,
that is, $\left(  \lambda-T\right)  \zeta=S\eta$. But $\operatorname{im}%
\left(  \lambda-T\right)  ^{\perp}=\left\langle \zeta_{\lambda}\right\rangle $
(see to the proof of Lemma \ref{lemH0H2}), which means that $\left(
S\eta,\zeta_{\lambda}\right)  =0$ or $\left(  \eta,S^{\ast}\zeta_{\lambda
}\right)  =0$. Notice that $S^{\ast}\zeta_{\lambda}=\zeta_{\lambda q}$. Since
$\ker\left(  \lambda-T\right)  =\left\{  0\right\}  $, the mapping $\left(
\zeta,\eta\right)  \mapsto\eta$ implements a topological isomorphism
$\ker\left(  l_{\lambda}^{1}\right)  \rightarrow\left\{  \zeta_{\lambda
q}\right\}  ^{\perp}$ onto. Once $\eta=\sum\beta_{n}e_{n}\in\left\{
\zeta_{\lambda q}\right\}  ^{\perp}$ one can easily restore $\zeta=\sum
\alpha_{n}e_{n}$ with
\[
\alpha_{n}=\frac{1}{\lambda^{n+1}}\left(  \eta-\eta_{n},\zeta_{\lambda
q}\right)  .
\]
Indeed, $\lambda\alpha_{0}=\beta_{0}$ and
\begin{equation}
\lambda\alpha_{n}-\alpha_{n-1}=\frac{1}{\lambda^{n}}\left(  \eta-\eta
_{n},\zeta_{\lambda q}\right)  -\frac{1}{\lambda^{n}}\left(  \eta-\beta
_{n}e_{n}-\eta_{n},\zeta_{\lambda q}\right)  =\frac{\beta_{n}}{\lambda^{n}%
}\left(  e_{n},\zeta_{\lambda q}\right)  =\frac{\beta_{n}}{\lambda^{n}}\left(
\lambda q\right)  ^{n}=q^{n}\beta_{n} \label{ab}%
\end{equation}
for all $n\geq1$, which means that $\left(  \lambda-T\right)  \zeta=S\eta$.
Notice that $\sum\left\vert \alpha_{n}\right\vert ^{2}<\infty$ holds automatically.

It remains to prove that for $\left(  \zeta,\eta\right)  \in\ker\left(
l_{\lambda}^{1}\right)  $ the inclusion $\left(  \zeta,\eta\right)
\in\operatorname{im}\left(  l_{\lambda}^{0}\right)  $ holds iff $\eta\in
M_{\lambda}$. But $\left(  \zeta,\eta\right)  \in\operatorname{im}\left(
l_{\lambda}^{0}\right)  $ iff $\zeta=S\theta$ for some $\theta=\sum\gamma
_{n}e_{n}$, that is, $q^{n}\gamma_{n}=\alpha_{n}$, $n\geq0$. Indeed, in this
case, using (\ref{ab}), we derive that
\begin{align*}
\left(  \lambda-q^{-1}T\right)  \theta &  =\lambda\gamma_{0}e_{0}+\sum
_{n\geq1}\left(  \lambda\gamma_{n}-q^{-1}\gamma_{n-1}\right)  e_{n}%
=\lambda\alpha_{0}e_{0}+\sum_{n\geq1}\left(  \lambda q^{-n}\alpha_{n}%
-q^{-n}\alpha_{n-1}\right)  e_{n}\\
&  =\lambda\alpha_{0}e_{0}+\sum_{n\geq1}q^{-n}\left(  \lambda\alpha_{n}%
-\alpha_{n-1}\right)  e_{n}=\sum\beta_{n}e_{n}=\eta.
\end{align*}
Further, $\zeta=S\theta$ is equivalent to
\begin{align*}
\gamma_{n}  &  =\alpha_{n}q^{-n}=\frac{1}{\lambda}\frac{1}{\left(
q\lambda\right)  ^{n}}\left(  \eta-\eta_{n},\zeta_{\lambda q}\right)
=-\frac{1}{\lambda}\frac{1}{\left(  q\lambda\right)  ^{n}}\left(  \eta
_{n},\zeta_{\lambda q}\right)  =-\frac{1}{\lambda}\frac{1}{\left(
q\lambda\right)  ^{n}}\sum_{k>n}\beta_{k}\left(  q\lambda\right)  ^{k}\\
&  =-q\sum_{k>n}\beta_{k}\left(  q\lambda\right)  ^{k-n-1}=-q\left(  \eta
_{n},T^{n+1}\zeta_{\lambda q}\right)
\end{align*}
with $\left\Vert \theta\right\Vert _{2}^{2}=\sum_{n}\left\vert \gamma
_{n}\right\vert ^{2}=\left\vert q\right\vert ^{2}\sum_{n}\left\vert \left(
\eta_{n},T^{n+1}\zeta_{\lambda q}\right)  \right\vert ^{2}=\left\vert
q\right\vert ^{2}\epsilon_{\lambda}\left(  \eta\right)  <\infty$.
\end{proof}

\subsection{The description of spectra}

Now we can summarize all results obtained in the present section.

\begin{theorem}
\label{thTSE}Let $\left(  T,S\right)  $ be a pair given by the unilateral
shift operator and the diagonal $q$-operator on the Banach space $\ell_{p}$,
$1\leq p<\infty$. Then%
\[
\mathbb{D}_{\left\vert q\right\vert ^{-1}}\backslash\mathbb{D}_{1}^{\circ
}\subseteq\sigma_{l}\left(  T,S\right)  \subseteq\mathbb{D}_{\left\vert
q\right\vert ^{-1}},\quad\sigma_{e}\left(  T,S\right)  =\left(  \partial
\mathbb{D}_{1}\cup\partial\mathbb{D}_{\left\vert q\right\vert ^{-1}}\right)
\times\left\{  0\right\}  ,
\]
where $\mathbb{D}_{1}^{\circ}$ is the interior of the disk $\mathbb{D}_{1}$.
In this case, for the left spectra we have
\begin{align*}
\sigma_{l}^{\pi,1}\left(  T,S\right)  \cap\left(  \mathbb{D}_{\left\vert
q\right\vert ^{-1}}\backslash\mathbb{D}_{1}^{\circ}\right)   &  =\Sigma
_{l}^{1}\left(  T,S\right)  \cap\left(  \mathbb{D}_{\left\vert q\right\vert
^{-1}}\backslash\mathbb{D}_{1}^{\circ}\right)  =\sigma_{l}^{\delta,1}\left(
T,S\right)  \cap\left(  \mathbb{D}_{\left\vert q\right\vert ^{-1}}%
\backslash\mathbb{D}_{1}^{\circ}\right)  =\mathbb{D}_{\left\vert q\right\vert
^{-1}}\backslash\mathbb{D}_{1}^{\circ},\\
\sigma_{l}^{\delta,2}\left(  T,S\right)  \cap\left(  \mathbb{D}_{\left\vert
q\right\vert ^{-1}}\backslash\mathbb{D}_{1}\right)   &  =\varnothing,
\end{align*}
and
\[
\sigma_{l}^{\pi,1}\left(  T,S\right)  \cap\mathbb{D}_{1}^{\circ}=\sigma
_{l}^{\delta,1}\left(  T,S\right)  \cap\mathbb{D}_{1}^{\circ}\text{,\quad
}\sigma_{l}^{\delta,2}\left(  T,S\right)  \cap\mathbb{D}_{1}^{\circ
}=\varnothing.
\]
The following equalities
\[
\sigma_{r}^{\pi,0}\left(  T,S\right)  =\varnothing,\quad\sigma_{r}^{\pi
,1}\left(  T,S\right)  =\sigma_{r}^{\pi,2}\left(  T,S\right)  =\sigma
_{r}\left(  T,S\right)  =\sigma_{r}^{\delta,i}\left(  T,S\right)  =\left\{
1\right\}
\]
for the right spectra hold for all $i$, $0\leq i\leq2$.
\end{theorem}

\begin{proof}
By Proposition \ref{propEx1}, we have $\mathbb{D}_{\left\vert q\right\vert
^{-1}}\backslash\mathbb{D}_{1}\subseteq\sigma_{l}\left(  T,S\right)
\subseteq\mathbb{D}_{\left\vert q\right\vert ^{-1}}$. But $\sigma_{l}\left(
T,S\right)  $ is a compact set (Proposition \ref{propNVs}), therefore the
inclusion $\mathbb{D}_{\left\vert q\right\vert ^{-1}}\backslash\mathbb{D}%
_{1}^{\circ}\subseteq\sigma_{l}\left(  T,S\right)  $ holds too (see also
\cite[Proposition 3.1]{DosAJM}).

The equalities for the left spectra $\sigma_{l}^{\pi,k}\left(  T,S\right)  $,
$\sigma_{l}^{\delta,k}\left(  T,S\right)  $ within $\mathbb{D}_{\left\vert
q\right\vert ^{-1}}\backslash\mathbb{D}_{1}$ can easily be driven from
Proposition \ref{propEx1}, whereas Lemma \ref{lemH0H2} implies the equalities
for the left spectra within $\mathbb{D}_{1}^{\circ}$. The equalities for the
right spectra $\sigma_{r}^{\pi,k}\left(  T,S\right)  $, $\sigma_{r}^{\delta
,k}\left(  T,S\right)  $ follow from Proposition \ref{propEx2}.

To prove the equality for the essential joint spectrum we use the functor $F$
of Buoni-Harte-Wickstead \cite{BHW} defined on the category of Banach spaces.
Recall that $F\left(  X\right)  =\ell_{\infty}\left(  X\right)  /k\left(
X\right)  $, where $k\left(  X\right)  $ is the subspace of all precompact
sequences on $X$. Since $F$ is annihilating over all compact operators, we
obtain that $F\left(  \mathcal{L}_{\lambda}\left(  T,S\right)  \right)  $ is
reduced to the following Banach space complex
\[%
\begin{tabular}
[c]{lllllll}%
$0\rightarrow$ & $F\left(  \ell_{p}\right)  $ & $\overset{F\left(  l_{\lambda
}^{0}\right)  =\left[
\begin{tabular}
[c]{l}%
$0$\\
$\lambda-q^{-1}F\left(  T\right)  $%
\end{tabular}
\ \ \right]  }{\longrightarrow}$ & $F\left(  \ell_{p}\right)  \oplus F\left(
\ell_{p}\right)  $ & $\overset{F\left(  l_{\lambda}^{1}\right)  =\left[
\begin{tabular}
[c]{ll}%
$\lambda-F\left(  T\right)  $ & $0$%
\end{tabular}
\ \ \right]  }{\longrightarrow}$ & $F\left(  \ell_{p}\right)  $ &
$\rightarrow0.$%
\end{tabular}
\
\]
It follows that $\left(  \lambda,0\right)  \in\sigma_{e}\left(  T,S\right)  $
iff $\lambda\in\sigma\left(  F\left(  T\right)  \right)  \cup q^{-1}%
\sigma\left(  F\left(  T\right)  \right)  $ (see \cite{Fain80}). But
$\sigma\left(  F\left(  T\right)  \right)  =\sigma_{e}\left(  T\right)
=\partial\mathbb{D}_{1}$ (see \cite[VII, 6.5]{Con}), therefore $\sigma
_{e}\left(  T,S\right)  \cap\mathbb{C}_{x}=\left(  \partial\mathbb{D}_{1}%
\cup\partial\mathbb{D}_{\left\vert q\right\vert ^{-1}}\right)  \times\left\{
0\right\}  $.

Finally, by Proposition \ref{propEx2}, $\sigma_{y}\left(  T,S\right)
=\left\{  \left(  0,1\right)  \right\}  $ and $\mathcal{R}_{1}\left(
T,S\right)  $ is a Fredholm complex. It follows that $\left(  0,1\right)
\notin\sigma_{e}\left(  T,S\right)  \cap\mathbb{C}_{y}$ or $\sigma_{e}\left(
T,S\right)  \cap\mathbb{C}_{y}=\varnothing$. Optionally, one can apply the
functor $F$ to the complex $\mathcal{R}_{1}\left(  T,S\right)  $, which turns
out to be an exact complex. Whence $\sigma_{e}\left(  T,S\right)  =\left(
\partial\mathbb{D}_{1}\cup\partial\mathbb{D}_{\left\vert q\right\vert ^{-1}%
}\right)  \times\left\{  0\right\}  $.
\end{proof}

\begin{remark}
It remains unclear whether $\sigma_{l}\left(  T,S\right)  \cap\mathbb{D}%
_{1}^{\circ}$ is empty or not. Notice that Lemma \ref{lemH1} provides more
complicated description of the cohomology group $H^{1}\left(  \lambda\right)
$ for $\lambda\in\mathbb{D}_{1}^{\circ}$.
\end{remark}

As a concluding remark, let us point out that the $q$-projection property from
the main Theorem \ref{lemSpLR} turns out to be the best possible projection
property for the operator $q$-planes that confirms Theorem \ref{thTSE}.

\end{document}